\documentclass[11pt,english,12p]{amsart}
\usepackage{amsmath}
\usepackage{amssymb}
\usepackage{tikz}
\usepackage[autostyle]{csquotes}
\usepackage{mathrsfs}
\usepackage{float}
\usepackage{tikz-cd} 
\usepackage{verbatim}
\usepackage{hyperref}
\usepackage{enumitem}
\usepackage{pst-node}
\usetikzlibrary{patterns}
\usetikzlibrary{arrows,positioning,shapes,fit,calc}

\usepackage{scalerel}
\usepackage{stackengine,wasysym}

 \usepackage[all,cmtip]{xy}

\usetikzlibrary{arrows,positioning,shapes,fit,calc}

\newtheorem{thm}{Theorem}[section]
\newtheorem{lem}[thm]{Lemma}
\newtheorem{prop}[thm]{Proposition}
\newtheorem{cor}[thm]{Corollary}

\newtheorem{defn}[thm]{Definition}

\newtheorem{remark}[thm]{Remark}

\newcommand{\R}{{\mathbb R}}

\newcommand{\Q}{{\mathbb Q}}

\newcommand{\C}{{\mathbb C}}

\newcommand{\End}{\textrm{End}}
\newcommand{\U}{\mathrm{U}}
\newcommand{\G}{\mathrm{G}}

\selectfont

\title[$\textrm{U}(n)$-structures and their induced minimal left ideals]{$\textrm{U}(n)$-structures  and their induced minimal left ideals}

\author[R. Su\'arez]{ Ricardo Su\'arez}
\address{Natural Science Division, Pepperdine University, Malibu, CA, USA}
\email{josericardo.suarez@pepperdine.edu}

\begin{document}
%\maketitle

\normalsize

\begin{abstract} 

In previous work,
%(see \cite{Su})  
we associated  to $\textrm{SU(3)}$, $\mathrm{G}_2$, and $\textrm{Spin(7)}$-structures minimal left ideals for the Clifford algebras $\mathbb{R}_{0,6},\mathbb{R}_{0,7}$, and $\mathbb{R}_{0,8}$, respectively. In this paper, we continue to analyze the link between Berger's classification theorem and the structure theorem of minimal left ideals for Clifford algebras of signature $(p,q)$ by  identifying $\mathrm{U}(n)$-structures with minimal left ideals for Clifford algebras  of various  signatures via the induced  Kahler polynomial $P(\omega_{0})$ associated with the symplectic form $\omega_{0}$ that defines the $\mathrm{U}(n)$-structure as a stabilizer subgroup of $\mathrm{O}(n)$.

%We say that a $G$-structure induces a minimal left ideal for a Clifford algebra of the form $\R_{p,q}$ if, by means of the vector space isomorphism between the space of exterior forms  $\bigwedge ^{*}(\R^{p+q})^{*}$ and the Clifford algebra $\R_{p,q}$, we can express the primitive idempotent $f$ that defines the minimal left ideal $\R_{p,q} f$ as the image of an algebraic expression of the model tensor that defines the special geometric structure in the space of exterior forms. In this paper we focus on special geometric structures in dimensions $6,7,$ and $8$ and the minimal left ideals that they induce for the Clifford algebras $\R_{0,6},\R_{0,7},$ and $\R_{0,8}$.
\end{abstract}
\maketitle
%\newpage
\section{Introduction}

In previous work (see \cite{Su}), we associated primitive idempotents that make up the minimal left ideals for the Clifford algebras $\R_{0,6}, \R_{0,7},$ and $\R_{0,8}$ with special 
geometric structures in dimensions $6,7,$ and $8$. In this paper, we continue our exploration of the geometric structures  in Marcel Berger's classification (see Theorem \ref{Berger}) by identifying $\textrm{U}(n)$-structures with  minimal left ideals of real Clifford algebras of signature $(p,q)$. %A distinct difference in the associations in the present work is that we associate $\textrm{U}(n)$-structures with minimal left ideals of the real Clifford algebras $\R_{p,q}$ where $p+q\ge 2n$, since not all of the generators $e_{1},\ldots,e_{p+q}$ are utilized in generating a primitive idempotent for the minimal left ideal for the Clifford algebra $\R_{p,q}$. For this reason, the exterior algebra and the Clifford algebra need not be generated from a Euclidean space of equal dimension.  
We say (see \cite{Su}) that a minimal left ideal is \textit{induced} by a $\G$-structure if  the minimal left ideal can be expressed as the image of an algebraic  expression involving the differential forms that define that $\G$-structure via the isomorphic extension of the quantization map $q^{*}:\bigwedge^{\bullet}(\R^{p+q})^{*}\xrightarrow{\cong} \R_{p,q}$. Conversely, a $\G$-structure can be \textit{recovered} from a minimal left ideal if the differential form that makes up the $\G$-structure is the image of an algebraic expression involving the primitive idempotent that defines that minimal left ideal via the isomorphic extension of the symbol map $\sigma^{*}:\R_{p,q}\xrightarrow{\cong} \bigwedge^{\bullet}(\R^{p+q})^{*}$. 

In the current work, we identify geometric structures with Clifford algebras  in cases where the underlying vector spaces for both algebras are not of the same dimension, so new terminology is needed. We therefore introduce the concept of a minimal left ideal being \textit{induced by an embedded $\G$-structure} when the minimal left ideal of the Clifford algebra $\R_{p,q}$ is induced by a $\G$-structure from a space $\R^{n}$ where $n<p+q$. At the same time, we say that a $\G$-structure is \textit{recovered by projection} when the $\G$-structure is recovered from an element $\tilde{f}$ in a Clifford algebra $\R_{p,q}$ such that $p+q>n$, rather than an algebraic expression involving the primitive idempotent (see Definition \ref{definition: recovered by projection}).  
%These new definitions allow us to relate geometric structures with Clifford algebras of higher dimensions, which is a property that an object like a spinor bundle on a $G$-manifold could have. 

We define a specific  $\textrm{U}(n)$-structure by fixing an almost complex structure $J_0$ with $J_0(e_{j})=e_{j+n}$ and $J_0(e_{j+n})=-e_{j}$ for all $j=1,2,\ldots,n$, which, in conjunction with the flat Riemannian metric $g_{0}$ on $\R^{2m}$, specifies the Kahler form $\omega_{0}$. 

In Lemma \ref{Standard Kahler form}, we give a local description of this  Kahler form, where  the association between the three tensors is given by the formula $\omega_{0}=g_{0}(J\cdot,\cdot)=e^1\wedge e^{n+1}+\cdots+ e^{n}\wedge e^{2n}.$ With the aid of the  Kahler form, we define what we call the \textbf{rational Kahler polynomial}, which we denote  $P^{\mathbb{Q}}(\omega_{0})$ (see Definition \ref{Rational Kahler}). In Theorem \ref{Kahler}, with the  aid of the rational Kahler polynomial, we show that our $\textrm{U}(p)$-structures on $\R^{2p}$ induce the  minimal left ideals  $\R_{p,p} f_{\textrm{U}(p)}$ for Clifford algebras $\R_{p,p}$, while for Clifford algebras $\R_{p,p+1}$ and $\R_{p,p+2}$, the minimal left ideals are induced by an embedded $\textrm{U}(p)$-structure on $\R^{2p}$. For all of these cases, the  primitive idempotent $f_{\textrm{U}(p)}$  is the image of the rational Kahler polynomial in $\R^{2p}$ via the extension of the quantization map; that is, $q^{*}(P^{\mathbb{Q}}(\omega_{0}))=f_{\textrm{U}(p)}$. Since the minimal left ideals $\R_{p,p+1} f_{\textrm{U}(p)}$ are  of complex type, and the minimal left ideals $\R_{p,p+2} f_{\textrm{U}(p)}$ are of quaternionic type, we go on to define complex and quaternionic structures that we can associate with these induced minimal left ideals. We conclude Section \ref{U(n) structures and their induced minimal left ideals}  with examples of $\textrm{U}(3)$-structures and their induced minimal left ideals in the Clifford algebras $\R_{3,3}, \R_{3,4},$ and $\R_{3,5}$. 

In Corollary \ref{Kahler Corrollary}, we show that for the  Clifford algebra $\R_{p,p}$, we can recover a $\textrm{U}(p)$-structure on $\R^{2p}$ from a minimal left ideal $\R_{p,p}f$, and we can projectively recover  a $\textrm{U}(p)$-structure on $\R^{2p}$ from a minimal left ideal $\R_{p,p+1}f$ (resp.\  $\R_{p,p+2}f$) for the Clifford algebra 
$\R_{p,p+1}$ (resp.\ $\R_{p,p+2}$),  where the primitive idempotent $f$ for all three cases generates the rational Kahler polynomial; that is,  $\sigma^{*}(f)=P^{\mathbb{Q}}(\omega_{0})$. We conclude this manuscript with Proposition \ref{Recovery theorem}, which states that for all other signatures, there exists an element $\tilde{f}\in\R_{p,q}$ which is a factor of a minimal left ideal that recovers a $\textrm{U}(m)$-structure; moreover, $\tilde{f}$  generates the Kahler polynomials in a subspace $\R^{2m}\subset \R^{p+q} $  where $m= \textrm{min}\{p,q\}.$  The corresponding  minimal left ideal  $\R_{p,q} h$ is then given by the Clifford product  $h=f e$ for some element $e\in\R_{p,q}$ such that $\sigma^{*}(h)=P^{\mathbb{Q}}(\tilde{\omega_{0}})\wedge\sigma^{*}(e)$, where $\sigma^{*}(e)$ can be viewed as an exterior form on a vector subspace $\R^{l}\subset \R^{p+q}$, where $l=p+q-2m$.

\section{Identifying $\G$-structures with minimal left ideals for Clifford algebras}

In this section, we provide an overview of what it means to identify a special geometric structure with a minimal left ideal of a Clifford algebra (see \cite{Su} for a more in-depth introduction).
We begin by defining a $\G$-structure for a general oriented Riemannian manifold $(M,g)$ of dimension $n$.
 
\begin{defn}
Let $\G$ be  a closed  Lie subgroup of $\mathrm{GL}(n,\R)$. A \textbf{$\G$-structure} on $M$ is a reduction of the structure group $\mathrm{GL}(n,\R)$ of the frame bundle $\mathcal{F}(M)$ to $\G$. That is, a $\G$-structure is a  principal sub-bundle $\mathcal{P}(M)$ of $\mathcal{F}(M)$ with fibre $\G$. If $\G$ is one of the groups that appear in the Berger classification theorem (see below),  then $\G$ is called a \textbf{special geometric structure}. 
\end{defn}

An important question pertaining to Riemannian manifolds is, which Lie subgroups of $\rm{GL(n,\R)}$ are also  holonomy groups of the Levi-Civita connection, and what conditions need to be imposed on the Riemannian manifold for this to be the case? Berger's theorem provides an answer to this question for any $n$-dimensional, simply connected manifold  with an irreducible, non-symmetric Riemannian metric (see \cite{Be, Jo, KN, Ra} for more on the concept of Riemannian holonomy).

\begin{thm}[Berger's Theorem]\label{Berger}

Suppose that $M$ is a simply connected manifold of dimension $n$, and that $g$ is an irreducible, non-symmetric Riemannian metric on $M$. Then exactly one of the following seven cases holds. 

\begin{enumerate}

\item $\mathrm{Hol}(g)=\mathrm{SO}(n)$.
\item $n=2m$ with $m\ge 2$, and $\mathrm{Hol}(g)=\rm{U(m)}$ in $\mathrm{SO(2m)}$.
\item $n=2m$ with $m\ge 2$, and $\mathrm{Hol}(g)=\mathrm{SU}(m)$ in $\mathrm{SO}(2m)$.
\item $n=4m$ with $m\ge 2$, and $\mathrm{Hol}(g)=\mathrm{Sp}(m)$ in $\mathrm{SO}(4m)$.
\item $n=4m$ with $m\ge 2$, and $\mathrm{Hol}(g)=\mathrm{Sp}(m)\mathrm{Sp}(1)$ in $\mathrm{SO}(4m)$.
\item $n=7$,  and $\mathrm{Hol}(g)=\mathrm{G_2}$ in $\mathrm{SO(7)}$.
\item $n=8$, and $\mathrm{Hol}(g)=\mathrm{Spin(7)}$ in $\mathrm{SO(8)}$.
\hfill \qed

\end{enumerate}
\end{thm}
Now via the general holonomy principle, we can identify each holonomy group on Berger's list with parallel tensor fields on the Riemannian manifold. That is, for each holonomy group in Berger's classification list, we can find  $\sigma_1, \ldots,\sigma_k\in\Gamma(\bigotimes T^{*}M\otimes \bigotimes TM)$ such that $\nabla^{g}\sigma_1= \cdots =\nabla^{g}\sigma_k=0$. It is well known that the local descriptions of these parallel tensor fields can be identified with model tensors in $\R^{n}$, and that the holonomy groups in Berger's list can be expressed as the  subgroups of $\textrm{GL}(n,\R)$ that stabilize  these model tensors (see \cite{Jo, KLL, KN, Ra}). Since every oriented Riemannian manifold $(M^{n},g)$ comes with a Riemannian metric $g$ and a volume form $dV_{g}$ that uniquely defines its orientation, and since in $\R^{n}$ the subgroup of $\textrm{GL}(n,\R)$ that stabilizes these two tensors is $\textrm{SO}(n)$, it is immediate that the holonomy group $\textrm{Hol}(g)$ is strictly smaller than $\rm{SO(n)}$  if and only if there  exist additional  nontrivial parallel tensor fields on $M^{n}$ with respect to  the Levi-Civita connection $\nabla^{g}$. Hence, for our purposes,  the holonomy groups in Berger's list can be thought of as stabilizer subgroups of $\mathrm{SO}(n)$ that stabilize additional tensors added to the standard inner product and volume form in $\R^{n}.$  

We now change our focus to Clifford algebras. (For elementary properties of Clifford algebras, see \cite{Ha, LM, Lo}). Let $V$ be a finite-dimensional $\R$ vector space with quadratic form $q$, and let $V^{\otimes}=\bigoplus_k V^{\otimes k}$ be its tensor algebra. We define the \textbf{Clifford algebra} of  the pair $(V,q)$ as the quotient of the tensor algebra by the two-sided ideal $I_q=\langle v\otimes v+q(x)1_{V^{\otimes}}\rangle$; that is, $C_{q}(V)=V^{\otimes}/I_q$. Choosing a basis for $(V,q)$, say $e_1,\ldots,e_n$,  we get the following canonical basis for the Clifford algebra generated by the following $2^n$ monomials:
\begin{equation*}
\label{canonicalbasis}
\{1_V\} \cup \{e_{i_1} \cdots e_{i_k} : i_1 < i_2 < \cdots < i_k,\ k = 1, \ldots , n\}.
\end{equation*}
Any quadratic space over $\R$ is isomorphic to a quadratic space of the form $\R^{p,q}$ with signature $(p,q)$, where $p$ is the number of positive definite generators and $q$ is the number of negative definite generators.\footnote{It should be clear from context when the symbol $q$ means the quadratic form, when it means the second part of the signature $(p,q)$, and when it means the quantization map.} The generators $e_1, \ldots, e_p$ are the positive definite generators, meaning that $q(e_{i})=1$ for $1 \leq i \leq p$; and the generators $e_{p+1}, \ldots, e_{p+q}$ are the negative definite generators, meaning that $q(e_{j})=-1$ for $p+1 \leq j \leq p+q$. For the quadratic spaces  $\R^{p,q}$, we denote the associated Clifford algebras by $\R_{p,q}$. 

It is well known that Clifford algebras are isomorphic to matrix algebras over the division algebras $\R,\C,$ and $\mathbb{H}$. This classification is given in the following theorem (see \cite{Po}).
\begin{thm}
\label{classtheorem}
The Clifford algebra $\R_{p,q}$, where $p+q=n$,  has the following minimal representations over $\R$, $\C$, and $\mathbb{H}$:

\begin{enumerate}
\item[i)] $\R_{p,q}\cong \R(2^{\frac{n}{2}})$ if $q-p = 0,6 \mod 8$.

\item[ii)] $\R_{p,q}\cong \C({2^{\frac{n-1}{2}}})$ if $q-p = 1,5 \mod 8$.

\item[iii)] $\R_{p,q}\cong \mathbb{H}({2^{\frac{n-2}{2}}})$ if $q-p = 2,4 \mod 8$.

\item[iv)] $\R_{p,q}\cong \mathbb{H}({2^{\frac{n-3}{2}}})\oplus \mathbb{H}({2^{\frac{n-3}{2}}})$ if $q-p = 3 \mod 8$.

\item[v)] $\R_{p,q}\cong \R({2^{\frac{n-1}{2}}})\oplus \R({2^{\frac{n-1}{2}}})$ if $q-p = 7 \mod 8$.

\end{enumerate}
\end{thm}

It  should be noted that some of the structure of $\R_{p,q}$ does not carry over to the matrix algebra with which it is associated. An alternative way to generate spinor modules, which generate our  representations of these real Clifford algebras, is to classify our Clifford algebras $\R_{p,q}$ in terms of their minimal left ideals, generated by primitive idempotents (see \cite{LW} for the theory behind these constructions). We begin by defining what a minimal left ideal is in the context of a Clifford algebra $\R_{p,q}$. 

\begin{defn}
For the Clifford algebra $\R_{p,q}$, any \textbf{minimal left ideal} is of the form $\R_{p,q} f$ where $f$ is a primitive idempotent: that is, an $f\in \R_{p,q}$ such that $f^{2}=f$ and $f$ cannot be written as the sum of two orthogonal idempotents. The minimal left ideals of the Clifford algebra  $\R_{p,q}$ are called \textbf{spinor spaces}, and the elements of the minimal left ideals are called \textbf{algebraic spinors}. In $\R_{p,q} f$, we define $f$ as the \textbf{identity unit spinor} that is identified with  $1\in\R\subset \R_{p,q}$.

\end{defn}
It is immediate that $\R_{p,q} f$ is a left $\R_{p,q}$ module, where module multiplication is given by $\R_{p,q}\times \R_{p,q} f\rightarrow \R_{p,q}  f $ via $(\phi,\psi f)\mapsto (\phi \psi) f$ for all $\phi\in\R_{p,q}$ and $\psi f\in\R_{p,q} f.$  The following theorem gives us the construction and classification of minimal left ideals in $\R_{p,q}$ (see \cite{LW}).

\begin{thm}\label{Minimal left ideals}
A minimal left ideal of $\R_{p,q}$ is of type $\R_{p,q} f$ where $f=\dfrac{1+e_{t_1}}{2} \cdots \dfrac{1+e_{t_k}}{2}$ is a primitive idempotent in $\R_{p,q}$  and $e_{t_1}, \ldots ,e_{t_k}$ is a set of commuting elements of the canonical basis such that $e_{t_i}^2=1$ for all $i=1, \ldots ,k=q-r_{q-p}$; moreover, the generators form a multiplicative   group of order $2^{q-r_{q-p}}$. The numbers $r_{q-p}$ are called the Radon-Hurwitz numbers,  given by the recurrence formula $r_0=0,\ r_1=1,\ r_2=2,\ r_3=2,\ r_j=3$ where $4\le j\le 7$, $r_{i+8}=r_i+4$ for $i\ge 0$, $r_{-1}=-1$, and $r_{-i}=1-i+r_{i-2}$ for $i\ge 2$.
\hfill \qed
\end{thm}

The minimal left ideals constructed in Theorem \ref{Minimal left ideals} are not only left $\R_{p,q}$ modules, but they also take the structure of right $f \,\R_{p,q} f$ modules, where this right module structure gives the minimal left ideal $\R_{p,q} f$ an $\R,\C$, or $\mathbb{H}$ vector space structure (see \cite{Lo}). In order to relate $\G$-structures with minimal left ideals of Clifford algebras, we establish a vector space isomorphism between the Clifford algebra $\R_{p,q}$ and the exterior algebra of $\R^{p+q}$, denoted $\bigwedge^{\bullet} \R^{p+q}$. This isomorphism is achieved via the assignment of the \textbf{quantization map} $q:\bigwedge^{\bullet} \R^{p+q}\rightarrow \R_{p,q}$, where  $e_{i_1}\wedge \cdots \wedge e_{i_k}\mapsto e_{i_1} \cdots e_{i_k}$. The inverse of this isomorphism is called the \textbf{symbol map}, which we denote by $\sigma: \R_{p,q}\rightarrow \bigwedge^{\bullet} \R^{p+q}$, given by $\sigma(e_{i_1} \cdots e_{i_k})=e_{i_1}\wedge \cdots \wedge e_{i_k}$.

These maps have natural extensions onto the spaces of exterior forms on $\R^{p+q}$. We denote these extensions by $q^{*}:\bigwedge^{\bullet}(\R^{p+q})^{*}\rightarrow \R_{p,q}$, where $q^{*}(e^{i_i}\wedge  \cdots \wedge e^{i_k})=e_{i_1} \cdots e_{i_k}$; and $\sigma^{*}:\R_{p,q}\rightarrow \bigwedge^{\bullet} (\R^{p+q})^{*}$, where $\sigma^{*}(e_{i_1} \cdots e_{i_k})=e^{i_1}\wedge \cdots \wedge e^{i_k}$.
Here $e^{1}, \ldots ,e^{n}$ is the dual basis to the canonical basis $e_1, \ldots ,e_n$.

 %With this form and the extension of the quantization map $q^{*}$, we go on to define the Clifford Hodge dual operator (see \cite{RV} for more on the Hodge duals and Clifford algebras).

%\begin{defn}
%For any $\psi\in \R_{p,q}$, its \textbf{Clifford Hodge dual} is defined as  $\star\psi=\tilde{\psi}\cdot q^{*}(dV)$, where $\tilde{\psi}$ is the reversion anti-involution of the Clifford algebra $\R_{p,q}$, induced from the reversion $\reallywidetilde{\hspace{1mm}v_1\otimes \cdots \otimes v_{n}} =v_{n}\otimes \cdots \otimes v_1$ for any  $v_1\otimes \cdots \otimes v_n\in (\R^{p,q})^{\otimes}$. 

%\end{defn}
We now define what it means for a minimal left ideal to be recoverable from a $\G$-structure.

\begin{defn}
Let $\phi_1, \ldots ,\phi_{k}\in \bigwedge^{\bullet}(\R^{p+q})^{*}$ define a $\G$-structure. We say that \textbf{the $\G$-structure induces a minimal left ideal} $\R_{p,q}  f$ in the Clifford algebra $\R_{p,q}$ if the primitive idempotent $f$ that defines that minimal left ideal can be recovered  from an  algebraic expression $\psi(\phi_1, \ldots ,\phi_k)$ in the exterior algebra $\bigwedge^{\bullet}(\R^{p+q})^{*}$ via the isomorphism induced from the quantization map; that is, $q^{*}(\psi(\phi_1, \ldots ,\phi_k))=f$.

\end{defn}
For the converse, we make the following definition. 

\begin{defn}
Let $f$ be a primitive idempotent  that defines a minimal left ideal $\R_{p,q}  f$, and let $\phi_1, \ldots ,\phi_{k}\in \bigwedge^{\bullet}(\R^{p+q})^{*}$ define a $\G$-structure. We say that \textbf{the $\G$-structure is recoverable from $\R_{p,q}  f$} if   $\phi_1, \ldots ,\phi_{k}$ can be written as algebraic expressions of the graded components of the   primitive idempotent in the Clifford algebra; that is, $\phi_1=\sigma^{*}(\psi_1(f)), \ldots, \phi_k=\sigma^{*}(\psi_k(f))$ for algebraic expressions $\psi_1(f), \ldots ,\psi_k(f)$ of $f$ in the Clifford algebra $\R_{p,q}$.

\end{defn}

In this paper, unlike in \cite{Su}, we do not assume that the dimension of the underlying real space of the Clifford algebra (where we assign a primitive idempotent) and the dimension of the underlying real space of our exterior algebra (where we identify the $\G$-structure) are equal. Hence, we now introduce the concept of an embedded induced minimal left ideal in order to identify a $\G$-structure with a minimal left ideal in the present case.

\begin{defn}\label{definition: induced embedded}
Let $\phi_1, \ldots ,\phi_{k}\in \bigwedge^{\bullet}(\R^{n})^{*}$ define a $\G$-structure. We say that \textbf{the $\G$-structure induces an embedded minimal left ideal} $\R_{p,q}  f$ in the Clifford algebra $\R_{p,q}$ if the primitive idempotent $f$ can be recovered  from an  algebraic expression  $\psi(\phi_1, \ldots ,\phi_k)$ in $\bigwedge^{\bullet}(\R^{n})^{*}$  via the  embedding of the extended quantization map $q^{*}:\bigwedge^{\bullet}(\R^{n})^{*}\hookrightarrow \R_{p,q}$, where $\bigwedge^{\bullet}(\R^{n})^{*}$ is isomorphic to its image in $\R_{p,q}$ for $p+q > n$.
 
\end{defn}

We can also extend this concept, where the dimension of the Clifford algebra is larger than the dimension of the exterior algebra, to $\G$-structures recovered by minimal left ideals; but instead of an embedding, this is done by restricting the symbol map to a Clifford subalgebra. 

\begin{defn}\label{definition: recovered by projection}
    We say that a $\G$-structure can be \textbf{recovered by projection} if there exists  an element $\tilde{f}$ in a subalgebra $A\subset \R_{p,q}$ (where $p+q>n$) such that $A$ is isomorphic to a Clifford algebra $\R_{r,s}$ (where $r+s=n$), where  the differential forms that makes up the $\G$-structure,  $\phi_{1},\ldots,\phi_{k}\in\bigwedge^{\bullet}(\R^{n})^{*}$, are the images of algebraic expressions $\psi_{1}(\tilde{f}),\ldots,\psi_{k}(\tilde{f})\in A$ via the extension of the symbol map restricted to the subalgebra $A$; that is, $\sigma^{*}\big|_{A}:A\xrightarrow{\cong} \bigwedge^{\bullet}(\R^{n})^{*}$, where $\sigma^{*}\big|_{A}(\psi_1(\tilde{f}))=\phi_{1},\ldots,\sigma^{*}\big|_{A}(\psi_k(\tilde{f}))=\phi_{k}$.

\end{defn}

 We now link the primitive idempotent and the $\G$-structures  recovered by projection. 

 \begin{defn}\label{definition: factor}
 
Let $A$ be a Clifford subalgebra of $\R_{p,q}$ such that $A \cong \R_{r,s}$ where $r+s=n$. Consider a $\G$-structure recovered by projection from $\tilde{f}\in A$. We say that $\tilde{f}$ is a \textbf{factor of the primitive idempotent} that  recovers a $\G$-structure in $\R^{n}$ if the differential forms that makes up the $\G$-structure,  $\phi_{1},\ldots,\phi_{k}\in\wedge^{\bullet}(\R^{n})^{*}$, are given by $\sigma^{*}\big|_{A}(\psi_1(\tilde{f}))=\phi_{1},\ldots,\sigma^{*}\big|_{A}(\psi_k(\tilde{f}))=\phi_{k}$, and the primitive idempotent $h$ that defines the minimal left ideal $\R_{p,q} h$ can be expressed as the product of $\tilde{f}$ and an additional element $e$ in $\R_{p,q}$, that is, $h=\tilde{f} e.$ 
\end{defn}
 
In the next section, we show how we can identify $\textrm{U}(n)$-structures with their induced minimal left ideals, or their induced embedded minimal left ideals, for Clifford algebras of various signatures.

%To summarize the primitive idempotent for Clifford algebras of signature not of the form $(p,p).(p,p+1),(p,p+2)$ can be viewed as inducing at $\textrm{U}(m)$ structure on a subset $\R^{2m}\subset \R^{p+q}$, and their primitive idempotents have the image in the exterior algebra $\bigwedge^{\bullet}(\R^{p+q})^{*}$ of the rational Kahler polynomial and the wedge product of some exterior form coming the subspace that does not have associated $\textrm{U}(m)$ structure, that is $\sigma^{*}(e)\in\bigwedge^{\bullet}(\R^{p+q-2m})^{*}$. 

\section{$\textrm{U}(n)$-structures and their induced minimal left ideals}\label{U(n) structures and their induced minimal left ideals}

The matrix group $\textrm{U}(n)$, known as the \textit{unitary group}, is composed of $n\times n$ complex matrices $U$ such that $U^{*}U=I_{n}$, where $U^{*}$ is the Hermitian conjugate of $U$. We can also view $\textrm{U}(n)$ as a subgroup of $\textrm{SO}(2n)$. In order to view $\textrm{U}(n)$ as a stabilizer subgroup of $\textrm{SO}(2n)$, we must first introduce a $\textrm{GL}(n,\C)$-structure on $\R^{2n}$, that is, an \textit{almost complex structure} $J_{0}\in \End(\R^{2m})$ such that $J_{0}^2=-\textrm{id}$.
  After a choice of basis (ours will be the canonical one), we can view $(\R^{2n},J_{0})$ as a complex vector space, where via this identification we reduce the structure group $\textrm{GL}(2n,\R)$ to $\textrm{GL}(n,\C).$ As stated in Section 2, $\textrm{SO}(2n)$ is the stabilizer of the model tensors $g_{0}=\sum_{j=1}^{2n} e^{i}\otimes e^{i}$, and the volume form $dV_{g_{0}}=e^{1}\wedge \cdots \wedge e^{2n}.$ Now if on $(\R^{2n}, J_{0})$ the complex structure is $g_{0}$-orthogonal, we can define a skew-symmetric form $\omega_{0}(\cdot\, ,\cdot):=g_{0}(J\cdot\, ,\cdot) \in \bigwedge^{2} (\R^{2n})^{*}$. This is known as the \textbf{standard Kahler form}. With the aid of the Kahler form on the vector space $(\R^{2n},J_{0})$, we  are able to define a Hermitian metric (a positive definite Hermitian inner product), via $h_{0}(\cdot\,,\cdot)=g_{0}(\cdot\,,\cdot)-i\omega_{0}(\cdot\, ,\cdot)$ whose stabilizer is the unitary group $\textrm{U}(n)$.  Moreover, from the description of the Hermitian metric $h_0$ as a sum of the real part $g_0$ and the imaginary part $\omega_0$, we can view $\mathrm{U}(n)$ as $\mathrm{U}(n)=\mathrm{O}(2n)\cap \mathrm{Sp}(2n)$ where $\mathrm{Sp}(2n)$ is the subgroup of $\textrm{SO}(2n)$ that stabilizes the symplectic form $\omega_0$. Now since, for the triple $(g_{0},J_{0},\omega_0)$, it is possible to generate one of the tensors from the other two, via the identification $\omega_{0}(\cdot,\cdot)=g(J\cdot,\cdot)$ we can identify with a  $\textrm{U}(n)$-structure with  stabilizer subgroup of the tensors  $J_{0}$ and $\omega_0$, that is, $\textrm{U}(n)=\textrm{GL}(n,\C)\cap \textrm{Sp} (2n,\R)$. 
 Hence, from now on, a $\textrm{U}(n)$-structure is identified with the parallel tensors $J_{0}, \omega_{0},$ and $g_{0}$ in $\R^{2n}$, and it is viewed as the stabilizer of any two of the three tensor fields.   
 
For our exploration in identifying $\mathrm{U}(n)$-structures with minimal left ideals of Clifford algebras, we fix a complex structure $J_{0}$ in $\R^{2n}$ such that on the standard basis $e_1, \ldots, e_{2n}$, we have $J_{0}(e_k)=e_{k+n}$ and $J_{0}(e_{k+n})=-e_{k}$ for all $k=1,2,\ldots ,n.$  This complex structure allows us to identify $\R^{2n}$ with $\C^{n}$ via $w_{k}=e_{k}+ie_{n+k}$ for $k=1,2,\ldots,n.$

\begin{lem}\label{Standard Kahler form}
For the complex structure $J_{0}$ such that $J_{0}(e_k)=e_{k+n}$ and $J_{0}(e_{k+n})=-e_{k}$ for all $k=1,2,\ldots,n$, the standard Kahler form can be locally defined in $\R^{2n}$ by the formula 
$\omega_{0}=e^{1}\wedge e^{n+1}+ \cdots +e^{n}\wedge e^{2n}$. 
\end{lem}

\begin{proof}
For the complex structure given by $J_{0}$ such that $J_{0}(e_k)=e_{k+n}$, and $J_{0}(e_{k+n})=-e_{k}$ for all $k=1,2,\ldots,n$, and the standard inner product in $\R^{2n}$ given by $g_{0}=\sum_{i=1}^{2n}e^{i}\otimes e^{i}$, we define the standard Kahler form via $\omega_{0}(\cdot\, ,\cdot)=g_{0}(J_{0}\cdot\,,\cdot).$ For any two vectors $v,w\in \R^{2n}$, we have $\omega_{0}(v,w)=g_{0}(J_0(v),w). $  Then we obtain 
\begin{align*}
\omega_{0}(v,w) & =g_{0}(J_0(v),w) = g_{0}(J_0\big(\sum_{\alpha=1}^{2n} v_{\alpha}\big),\sum_{\beta=1}^{2n}w_{\beta}) \\
& = g_{0}\big(\sum_{\alpha=1}^{n}v^{\alpha}e_{\alpha+n}-v^{\alpha+n}e_{\alpha},\sum_{\beta=1}^{2n}w^{\beta}e_{\beta}\big) \\
& = \sum_{\alpha=1}^{n} \big\{g_{0}(v^{\alpha}e_{\alpha+n},w^{\alpha+n}e_{\alpha+n})-g_{0}(v^{\alpha+n}e_{\alpha},w^{\alpha}e_{\alpha})\big\} \\ 
& = \sum_{\alpha=1}^{n} (v^{\alpha} w^{\alpha+n}-v^{\alpha+n}w^{\alpha}) \\
& = \sum_{\alpha=1}^{n} e^{\alpha}\otimes e^{\alpha+n}-e^{\alpha+n}\otimes e^{\alpha}(v,w) \\
& = \sum_{\alpha=1}^{n} e^{\alpha}\wedge e^{\alpha+n}(v,w).
\end{align*}
Therefore the local description of this symplectic form associated with $J_{0}$ is $\omega_{0}=e^{1}\wedge e^{1+n}+ \cdots +e^{n}\wedge e^{2n}$. It follows that $d\omega_{0}=0$; hence the symplectic form $\omega_0$ is Kahler.
\end{proof}

Now if we take the wedge product of the Kahler form with itself, we obtain $\omega_{0}\wedge\omega_{0}=(e^{1}\wedge e^{1+n}+ \cdots +e^{n}\wedge e^{2n})\wedge(e^{1}\wedge e^{1+n}+ \cdots +e^{n}\wedge e^{2n})=\sum_{i\neq j} (e^{i}\wedge e^{i+n})\wedge (e^{j}\wedge e^{j+n})$.\ Since  $(e^{i}\wedge e^{i+n})\wedge (e^{j}\wedge e^{j+n})= (e^{j}\wedge e^{j+n})\wedge (e^{i}\wedge e^{i+n})$ for $i\neq j$, we get can rewrite the sum as  

\[\omega_{0}\wedge \omega_{0}=2\sum_{i<j}^n e^{i}\wedge e^{i+n}\wedge e^j\wedge e^{j+n}.\]

Modifying the notation to write $\omega_{0}^2:=\omega_{0}\wedge \omega_{0}$, we have

 $$ \dfrac{\omega_{0}^2}{2}=\sum_{i<j}^n e^{i}\wedge e^{i+n}\wedge e^j\wedge e^{j+n}.$$ 
Extending this analysis to  products of the  form $e^{i_1}\wedge e^{i_1+n}\wedge \cdots\wedge e^{i_{m}}\wedge e^{i_{m}+n}$, where $1\leq i_{k}\leq n$ for $k=1,2,\ldots,m$, we make use of the following two properties.
 \begin{enumerate}
    \item   If, for any $k,j$ such that  $1\leq k,j \leq m \leq n$ and  $k\neq j$, we have $i_{k}=i_{j}$, then the entire product is zero; that is, $e^{i_1}\wedge e^{i_1+n}\wedge \cdots\wedge e^{i_{m}}\wedge e^{i_{m}+n}=0.$
    \item Since all products of the form $e^{i_j} \wedge e^{i_j + n}$ commute, we can rearrange the product $e^{i_1}\wedge e^{i_1+n}\wedge \cdots\wedge e^{i_{m}}\wedge e^{i_{m}+n}$ in $m!$ different ways. 
 \end{enumerate}
 Now for an $m$ product, $m\leq n$, of Kahler forms, we obtain 
\[\omega_{0}^m=\sum_{ 1\leq i_1,\ldots,i_{m}\leq n}%\atop\scriptstyle i_1\neq i_2\cdot \neq i_{m}}
e^{i_1}\wedge e^{i_1+n}\wedge \cdots\wedge e^{i_{m}}\wedge e^{i_{m}+n}.\]
By property $(1)$, it is immediate that only nonzero wedge products are associated  to the indices where $i_1,\ldots,i_{m}$ are all distinct from one another. Moreover, from $(2)$, we can rearrange all of the nonzero sums in the same sequential order, namely one where $1\leq i_1<\cdots<i_{m}\leq n.$ From this we can immediately see that 
\[\omega_{0}^m=\sum_{ 1\leq i_1<\cdots <i_{m}\leq n}%\atop\scriptstyl
 % i_1\neq i_2\cdot \neq i_{m}}
m! e^{i_1}\wedge e^{i_1+n}\wedge \cdots\wedge e^{i_{m}}\wedge e^{i_{m}+n},\] or, equivalently,
\[\dfrac{\omega_{0}^m}{m!}=\sum_{ 1\leq i_1<\cdots <i_{m}\leq n}%\atop\scriptstyl
 % i_1\neq i_2\cdot \neq i_{m}}
 e^{i_1}\wedge e^{i_1+n}\wedge \cdots\wedge e^{i_{m}}\wedge e^{i_{m}+n}.
 \]

With $\dfrac{\omega_{0}^{m}}{m!}$ well-defined for $0\leq m\leq n$, where for the  $m=0$ case we define  $\omega_{0}^{0}=1$, we can now define an all one polynomial in the Kahler form $\omega_{0}$, which we state in the following lemma.

\begin{lem}
For the standard Hermitian form in $\R^{2n}$ locally defined in $\R^{2n}$ by the formula $\omega_{0}=e^{1}\wedge e^{1+n}+ \cdots +e^{n}\wedge e^{2n}$, we define the Kahler polynomial associated with $\omega_{0}$ as the differential form
\[
P(\omega_{0})=\sum_{m=0}^{n}\dfrac{\omega_{0}^m}{m!}.
\]
The Kahler polynomial $P(\omega_{0})$ is, in fact, an all one polynomial, as the leading coefficients of all of the terms are $1$. \hfill \qed
\end{lem}

The  expanded form of the Kahler  polynomial is given by 
\[
P(\omega_{0})= 1 + \sum_{m=1}^{n} \, \sum_{ 1\leq i_1<\cdots <i_{m}\leq n}%\atop\scriptstyl
 % i_1\neq i_2\cdot \neq i_{m}}
 e^{i_1}\wedge e^{i_1+n}\wedge \cdots\wedge e^{i_{m}}\wedge e^{i_{m}+n}.
 \]

In order for us to relate the $\textrm{U}(n)$-structures to minimal left ideals for certain Clifford algebras of signature $(p,q)$, we modify the definition of the Kahler polynomial slightly. 

\begin{defn}\label{Rational Kahler}
For the Kahler polynomial $P(\omega_{0})$ associated with the Kahler form $\omega_{0}$ in $\R^{2n}$, we define the \textbf{rational Kahler polynomial} by $P^{\mathbb{Q}}(\omega_{0})=\dfrac{1}{2^{n}}P(\omega_{0}).$
\end{defn}

The rational Kahler polynomial $P^{\mathbb{Q}}(\omega_{0})$ provides the necessary algebraic expression in order for our $\textrm{U}(n)$-structure to induce a minimal left ideal. 

\begin{thm}\label{Kahler}
The $\textrm{U(p)}$-structure $(J_{0}\,\omega_{0},g_{0})$ in $\R^{2p}$, where the complex structure is given by $J_{0}(e_k)=e_{k+p}$ and $J_{0}(e_{k+p})=-e_{k}$ for $k=1,2,\ldots,p$, and where the  Kahler form is locally defined by the formula 
$\omega_{0}=e^{1}\wedge e^{1+p}+ \cdots +e^{p}\wedge e^{2p}$, induces a minimal left ideal on the Clifford algebras $\R_{p,p}$, and induces an embedded minimal left ideal for the Clifford algebras  $\R_{p,p+1}$ and $\R_{p,p+2}$ for $p\geq 1$. The defining primitive idempotent in terms of the $\textrm{U}(p)$-structure is   given by $f_{\textrm{U}(p)}=q^{*}(P^{\mathbb{Q}}(\omega_{0}))$. 
\end{thm}

\begin{proof}
We begin with a $\textrm{U}(p)$-structure $(J_{0},\omega_{0})$ on $\R^{2p}$ where the complex structure is given by $J_{0}(e_k)=e_{k+p}$ and $J_{0}(e_{k+p})=-e_{k}$, for $k=1,2,\ldots,p$. The Kahler form   $\omega_{0}=e^{1}\wedge e^{1+p}+ \cdots +e^{p}\wedge e^{2p}$ then defines the rational Kahler polynomial $P^{\mathbb{Q}}(\omega_{0})\in\bigwedge^{\bullet}(\R^{2p})^{*}$. 
Shifting our attention to  Clifford algebras of the form $\R_{p,p}$, $\R_{p,p+1}$, and $\R_{p,p+2}$ where $p\ge 1$, we introduce  the index notation $\sigma_{j}=(j,p+j)$ for $j=1,2,\ldots,p$, where we have, for any $1\leq i <j\leq p$ the commutation property $e_{\sigma_{i}}e_{\sigma_{j}}=e_{\sigma_{j}}e_{\sigma_{i}}$. 
 From this, we define the following element for $m\leq p$:
\[f_{m}= 1 + \sum_{k=1}^{m}\sum_{1\leq i_1<\cdots <i_{k}\leq p} e_{\sigma_{i_1}}\cdots e_{\sigma_{i_{k}}}.
\]
When $m=p$, we denote the element  $\dfrac{1}{2^{p}}f_{p}$ by $f_{\textrm{U}(p)}$. 

Now the  rational Kahler polynomial given by $\omega_{0}$ in $\R^{2p}$ is explicitly given by the formula
\[
P^{\mathbb{Q}}(\omega_{0})=\dfrac{1}{2^{p}} \left( 1+\sum_{m=1}^{p} \, \sum_{ 1\leq i_1<\cdots <i_{m}\leq p}%\atop\scriptstyl
 % i_1\neq i_2\cdot \neq i_{m}}
 e^{i_1}\wedge e^{i_1+p}\wedge \cdots\wedge e^{i_{m}}\wedge e^{i_{m}+p} \right).
 \]
Applying  the extended quantization map, $q^{*}:\bigwedge^{\bullet}(\R^{2p})^{*}\rightarrow \R_{p,p}$, to the rational Kahler polynomial, we obtain  
\begin{align*}
q^{*}(P^\Q(\omega_{0})) & =\dfrac{1}{2^{p}} \left( 1 +  \sum_{m=1}^{p} \, \sum_{ 1\leq i_1<\cdots <i_{m}\leq p}%\atop\scriptstyl
 % i_1\neq i_2\cdot \neq i_{m}}
 q^*(e^{i_1}\wedge e^{i_1+p}\wedge \cdots\wedge e^{i_{m}}\wedge e^{i_{m}+p}) \right) \\
& = \dfrac{1}{2^{p}} \left( 1+ \sum_{m=1}^{p} \, \sum_{ 1\leq i_1<\cdots <i_{m}\leq p}%\atop\scriptstyl
 % i_1\neq i_2\cdot \neq i_{m}}
 e_{i_1} e_{i_1+p}\cdots  e_{i_{m}} e_{i_{m}+p} \right) \\
&  = \dfrac{1}{2^{p}} \left( 1 + \sum_{m=1}^{p} \, \sum_{ 1\leq i_1<\cdots <i_{m}\leq p}%\atop\scriptstyl
 % i_1\neq i_2\cdot \neq i_{m}}
 e_{\sigma_1}\cdots e_{\sigma_{m}} \right)\\
& = \dfrac{1}{2^{p}}f_{p}=:f_{\textrm{U}(p)}.
\end{align*}
Moreover, if we use the extended quantization map $q^{*}$ to embed the algebra $\bigwedge^{\bullet}(\R^{2p})^{*}$ into its isomorphic image in $\R_{p,p+1}$ and $\R_{p,p+2}$, we again have the result $q^{*}(P^{\mathbb{Q}}(\omega_{0}))=f_{\textrm{U}(p)}$. Thus we can conclude that $q^{*}(P^{\mathbb{Q}}(\omega_{0}))=f_{\textrm{U}(p)}$ in the Clifford algebras $\R_{p,p},\R_{p,p+1},$ and $\R_{p,p+2}$.
To show that these are minimal left ideals, we begin with the fact that we can  decompose  $f_{p}$ in the following manner:
\[
f_{p}=(f_{p-1}+f_{p-1}e_{\sigma_{p}}),
\]
Where the expanded form is
\[
f_{p}= \left( 1 + \sum_{k=1}^{p-1}\sum_{1\leq i_1<\cdots <i_{k}\leq p-1} e_{\sigma_{i_1}}\cdots e_{\sigma_{i_{k}}} \right) + \left( 1 + \sum_{k=1}^{p-1}\sum_{1\leq i_1<\cdots <i_{k}\leq p-1} e_{\sigma_{i_1}}\cdots e_{\sigma_{i_{k}}} \right) e_p.
\]
\begin{comment}
$$ f_{p}= (1+e_{\sigma_1}+\cdots+e_{\sigma_{p-1}}+e_{\sigma_{1}}\cdot e_{\sigma_{2}}+\cdots+e_{\sigma_{p-2}}\cdot e_{\sigma_{p-1}}+\cdots +e_{\sigma_1}\cdots e_{\sigma_{p-1}})+$$

$$(e_{\sigma_{p}}+e_{\sigma_1}\cdot e_{\sigma_{p}}+\cdots+e_{\sigma_{p-1}}\cdot e_{\sigma_{p}}+e_{\sigma_{1}}\cdot e_{\sigma_{2}}\cdot e_{\sigma_{p}}+\cdots+e_{\sigma_{p-2}}\cdot e_{\sigma_{p-1}}\cdot e_{\sigma_{p}}+\cdots +e_{\sigma_1}\cdots e_{\sigma_{p-1}}\cdot e_{\sigma_{p}}).$$
\end{comment}
Thus, we can decompose the image $q^{*}(P^{\mathbb{Q}}(\omega_{0}))=f_{\textrm{U}(p)}$ as $q^{*}(P^{\mathbb{Q}}(\omega_{0}))=\dfrac{1}{2^{p}}(f_{p-1}+f_{p-1}e_{\sigma_{p}}).$ 
  Now if we split off  the $e_{p-1}$ term in $f_{p-1}$ in an identical manner, we obtain $f_{p-1}=f_{p-2}+f_{p-2} e_{\sigma_{p-1}}$, from which  we obtain $f_{\textrm{U}(p)}=\dfrac{1}{2^{p}}f_{p-1}(1+e_{\sigma_{p}})=\dfrac{1}{2^{p}}f_{p-2}(1+e_{\sigma_{p-1}})(1+e_{\sigma_{p}}).$
  Continuing with this process until we get the full factorization for $q^{*}(P^{\mathbb{Q}}(\omega_{0}))=f_{\textrm{U}(p)}$ gives us the following: 
\[f_{\textrm{U}(p)}=\dfrac{1}{2^{p}}\prod_{i=1}^p (1+e_{\sigma_{i}})=\prod _{i=1}^p \big(\dfrac{1+e_{\sigma_{i}}}{2}\big).
\]
As we established above, all of the $e_{\sigma_{i}}$ for $i=1,2,3,\ldots,p$ are commuting involutions in the Clifford algebras $\R_{p,p}, \R_{p,p+1},$ and $\R_{p,p+2}$; and from Theorem \ref{Minimal left ideals}, we can see that for these Clifford algebras the number of commuting involutions needed to generate a primitive idempotent  is $p$.
     We provide the calculation of this fact in Table \ref{commuting involution table}.

\begin{table}[ht] %\label{commuting involution table}
    \centering
    \caption{The number of commuting involutions to generate a minimal left ideal}\label{commuting involution table}
    \begin{tabular}{|l|l|}
        \hline
        The Clifford algebra  & $k=q-r_{q-p}$\\
        \hline
       $\R_{p,p}$& $k=p-r_{0}=p$ \\
        \hline
        $\R_{p,p+1}$ &  $k=p-r_{1}=p+1-1=p$\\
        \hline
        $\R_{p,p+2}$ & $k=p-r_{2}=p+2-2=p$ \\
        \hline
    \end{tabular}
\end{table}

Thus it follows that  $f_{\textrm{U}(p)}$ is a primitive idempotent for $\R_{p,p}, \R_{p,p+1},$ and $\R_{p,p+2}$. Hence the $\textrm{U}(p)$-structure on $\R^{2p}$ induces a minimal left ideal for the Clifford algebras $\R_{p,p}$, and induces an embedded minimal left ideal for the Clifford algebras  $\R_{p,p}, \R_{p,p+1},$ and $\R_{p,p+2}.$

\end{proof}
Since $\textrm{U}(n)$ contains $\textrm{SU}(n)$ as a subgroup, we have that $\textrm{SU(n)}$, when viewed as a stabilizer subgroup of $\textrm{SO}(2n)$, will be the stabilizer of the same tensor fields as $\textrm{U}(n)$ with one additional structure. That is, it is also the stabilizer of the complex volume form $\Omega=dz^1\wedge\cdots\wedge dz^n$, where the complex coordinates in our construction are given by $dz^{j}=e^{j}+i  e^{p+j}$ for $j=1,\ldots,n.$ 
Now  since $\textrm{U(p)}$-structures on $\R^{2p}$ induce minimal left ideals (resp.\ embedded minimal left ideals) on the Clifford algebras $\R_{p,p}$ (resp.\ $\R_{p,p+1}$ and $\R_{p,p+2}$) with the aid of the rational Kahler polynomial, it is immediate that $\textrm{SU}(p)$-structures on $\R^{2p}$ also induce minimal left ideals (resp.\ embedded minimal left ideals)  on Clifford algebras $\R_{p,p}$ (resp.\ $\R_{p,p+1}$ and $\R_{p,p+2}$) with the aid of the same rational Kahler polynomial. Thus, when using an $\textrm{SU}(p)$-structure to generate minimal left ideals, we make no immediate use of  the complex volume form. We therefore have the following corollary.

\begin{cor}
Let the complex structure $J_0$ and the Kahler form $\omega_0$ define a $U(p)$-structure that induces a minimal left ideal (resp.\ an embedded minimal left ideal) for the Clifford algebras $\R_{p,p}$ (resp.\ $\R_{p, p+1}$ and $\R_{p,p+2}$) as in Theorem \ref{Kahler}. Then any $SU(p)$-structure where the complex structure and Kahler form coincide with the $U(p)$-structure also induces a minimal left ideal (resp.\ embedded minimal left ideal) for the Clifford algebra $\R_{p,p}$ (resp.\ $\R_{p,p+1}$ and $\R_{p,p+2}$). Moreover, the primitive idempotent that defines the minimal left ideal is given by $f_{\textrm{SU}(p)}=q^{*}(P^{\mathbb{Q}}(\omega_{0})).$ 
\hfill \qed
\end{cor}

As is well known (see \cite{Lo}), the minimal left ideals we described, $\R_{p,s} f_{\textrm{U}(p)}$ for $s=p, p+1, p+2$, are themselves right $\mathfrak{D}=f_{\textrm{U}(p)}\R_{p,s} f_{\textrm{U}(p)}$-modules. The  minimal left ideals $\R_{p,p} f_{\textrm{U}(p)}, \R_{p,p+1} f_{\textrm{U}(p)},$ and $\R_{p,p+2} f_{\textrm{U}(p)}$  can be seen as vector spaces of real dimension  $2^{p}, 2^{p+1},$ and $2^{p+2}$ respectively, where the minimal left ideal representations can be given  by the left multiplication  endomorphism:
\[
\begin{aligned}
     L:\R_{p,s}\rightarrow \textrm{End}_{\mathfrak{D}}(\R_{p,s} f_{\textrm{U}(p)}).
 \   
\end{aligned}
\]
Given the matrix algebra isomorphism for each Clifford algebra, it is immediate that $\R_{p,p} f_{\textrm{U}(p)}$ is of real type, $\R_{p,p+1} f_{\textrm{U}(p)}$ is of complex type, and $\R_{p,p+2} f_{\textrm{U}(p)}$ is of quaternionic type. Moreover, it is well known (see \cite{Lo}) that $\mathfrak{O}$ for these Clifford algebras takes the isomorphisms shown in Table \ref{table: isomorphisms}.
\begin{table}[ht]
    \centering
    \caption{The structure of $\mathfrak{D}$ for $\R_{p,p}$, $\R_{p, p+1}$, and $\R_{p, p+2}$}\label{table: isomorphisms}
    \begin{tabular}{|l|l|}
        \hline
        The Clifford algebra $\R_{p,q}$  & $\mathfrak{D}=f_{\textrm{U}(p)} \R_{p,q} f_{\textrm{U}(p)}$\\
        \hline
       $\R_{p,p}$& $\mathfrak{D}\cong \R$ \\
        \hline
        $\R_{p,p+1}$ & $\mathfrak{D}\cong \C$\\
        \hline
        $\R_{p,p+2}$ & $\mathfrak{D}\cong \mathbb{H}$ \\
        \hline
    \end{tabular} 
\end{table}
From Table \ref{table: isomorphisms}, it is clear that the induced minimal left ideals as right $\mathfrak{D}$-modules can be thought of as right  $\R, \C,$ and $\mathbb{H}$ modules respectively. Since $\R_{p,p}$ is a Clifford algebra of real type, the minimal left ideal  $\R_{p,p} f_{\textrm{U}(p)}$ will yield the real representation $\textrm{End}_{\mathfrak{D}}(\R_{p,p} f_{\textrm{U}(p)})\cong \R(2^{p}).$  The Clifford algebra $\R_{p,p+1}$, on the other hand,  is of complex type.  Moreover,  since the Clifford algebra $\R_{p,p+1}$  is an odd  Clifford algebra, the  pseudoscalar  $\Gamma=e_{1}\cdots e_{2p+1}$  is negative definite and central for this signature (see \cite{DL} for more information on the pseudoscalar). That is,  we can use $\Gamma$ as our complex structure that gives our minimal left ideal a complex vector space structure,  where $(a+bi) h f_{\textrm{U}(p)}=ah  f_{\textrm{U}(p)}+\Gamma h  f_{\textrm{U}(p)}. $ Hence, with the aid of $\Gamma$, we can define a complex basis on the minimal left ideal $\R_{p,p+1} f_{\textrm{U}(p)}$, where left multiplication by $\Gamma$ will coincide with right multiplication by $f_{\textrm{U}(p)} \Gamma f_{\textrm{U}(p)} $. This is immediate since for any $h f_{\textrm{U}(p)}\in \R_{p,p+1} f_{\textrm{U}(p)} $, we have 
\begin{align*}
  (h f_{\textrm{U}(p)}) (f_{\textrm{U}(p)}\Gamma f_{\textrm{U}(p)})  &= h (f_{\textrm{U}(p)})^2\Gamma f_{\textrm{U}(p)}  \\
   & = h f_{\textrm{U}(p)}\Gamma f_{\textrm{U}(p)}  \\
   & = h \Gamma ( f_{\textrm{U}(p)})^2  \\
  & =h \Gamma   f_{\textrm{U}(p)} \\
   & = \Gamma h  f_{\textrm{U}(p)}.\\
\end{align*}
It is immediate  that $\Gamma$, which will define our complex structure via the left multiplication endomorphism, will also define the right $\mathfrak{O}$-structure on
$\R_{p,p+1}  f_{\textrm{U}(p)}$; and with this choice of complex structure, we obtain matrix representations  $\textrm{End}_{\mathfrak{D}}(\R_{p,p+1}  f_{\textrm{U}(p)})\cong \C(2^{p}).$

Lastly, the Clifford algebra $\R_{p,p+2}$ is of quaternionic type, and we have a convenient choice for the quaternionic structure on the minimal left ideal $\R_{p,p+2}  f_{\textrm{U}(p)}$.  Notice that $e_{2p+1}, e_{2p+2},$ and $e_{2p+1} \cdot e_{2p+2}$ define a quaternion structure on $\R_{p,p+2}$ in the sense that the left action by $a+bi+cj+dk\in\mathbb{H}$ can be identified with a left action by $a+be_{2p+1}+ce_{2p+2}+de_{2p+1} \cdot e_{2p+2}$. All of these elements commute with the  primitive idempotent $ f_{\textrm{U}(p)}$, and hence are the elements in $\mathfrak{O}$ that define the isomorphism $\mathfrak{O}\cong \mathbb{H}$. This choice of quaternionic structure, giving us an $\mathbb{H}$ basis on  $\R_{p,p+2}  f_{\textrm{U}(p)}$ that also allows us to define a right $\mathfrak{O}$ module structure on $\R_{p,p+2}   f_{\textrm{U}(p)}$, will provide us with the isomorphism $\textrm{End}_{\mathfrak{D}}(\R_{p,p+2}  f_{\textrm{U}(p)})\cong \mathbb{H}(2^{p}).$ 

We conclude this section with three examples that relate to $\textrm{U}(3)$-structures.

\subsection{$\textrm{U}(3)$-structure and the induced minimal left ideal in the Clifford algebra $\R_{3,3}$}
We begin with  a $\textrm{U(3)}$-structure in  $\R^{6}$ given by a Kahler form in $\R^{6}$ that is described locally  by the equation $\omega_{0}=e^{14}+e^{25}+e^{36}$,  and a complex structure  $J_{0}\in\textrm{End}(\R^{6})$ completely  determined by the Kahler form $\omega_{0}$ and the standard inner product in $\R^{6}$ via the relation $\omega_{0}(\cdot \,,\cdot)=\langle J_{0}(\cdot),\cdot \rangle.$ 

The rational Kahler polynomial is then given by
\begin{align*}
P^{\mathbb{Q}}(\omega_{0}) & =\frac{1}{8}(1+\omega_{0}+\frac{\omega_{0}\wedge\omega_{0}}{2}+\frac{\omega_{0}^{3}}{6}) \\
& =\frac{1}{8}(1+e^{14}+e^{25}+e^{36}+e^{1425}+e^{1436}+e^{2536}+e^{142536})\in\bigwedge^{\bullet}(\R^{6})^{*}.
\end{align*}
From the extension of the quantization map $q^{*}:\bigwedge^{\bullet}(\R^{6})^{*}\rightarrow\R_{3,3}$, we send the rational Kahler polynomial into the expression  $q^{*}(P^{\mathbb{Q}}(\omega_{0}))=\dfrac{1}{8}(e_{14}+e_{25}+e_{36}+e_{1425}+e_{1436}+e_{2536}+e_{142536})$, which decomposes as the product to $(\frac{1+e_{14}}{2})  (\frac{1+e_{25}}{2}) (\frac{1+e_{36}}{2})$ in the Clifford algebra $\R_{3,3}$. By Theorem \ref{Minimal left ideals}, it is immediate that this is a primitive idempotent, which we denote $f_{\textrm{U}(3)}$, that gives us the minimal left ideal $\R_{3,3} f_{\textrm{U}(3)}$. 

This minimal left ideal is an $8$-dimensional real vector space given by the following basis elements: 
\begin{enumerate}
    \item  $f_{\textrm{U}(3)}= \frac{1}{8}(1 + e_{14} + e_{25} + e_{36} - e_{1245} - e_{1346} - e_{2356} - e_{123456})$
    \item $e_1 f_{\textrm{U}(3)} = \frac{1}{8}(-e_{12356} - e_{23456} + e_{125} + e_{136} - e_{245} - e_{346} + e_1 + e_4)$
    \item $e_2 f_{\textrm{U}(3)} = \frac{1}{8}(e_{12346} + e_{13456} - e_{124} + e_{145} + e_{236} - e_{356} + e_2 + e_5)$
    \item $e_3 f_{\textrm{U}(3)} = \frac{1}{8}(-e_{12345} - e_{12456} - e_{134} + e_{146} - e_{235} + e_{256} + e_3 + e_6)$
    \item $e_{12} f_{\textrm{U}(3)} = \frac{1}{8}(e_{1236} - e_{1356} + e_{2346} + e_{3456} + e_{12} + e_{15} - e_{24} + e_{45})$
    \item $e_{13} f_{\textrm{U}(3)} = \frac{1}{8}(-e_{1235} + e_{1256} - e_{2345} - e_{2456} + e_{13} + e_{16} - e_{34} + e_{46})$
    \item $e_{23} f_{\textrm{U}(3)} = \frac{1}{8}(e_{1234} - e_{1246} + e_{1345} + e_{1456} + e_{23} + e_{26} - e_{35} + e_{56})$
    \item $e_{123} f_{\textrm{U}(3)} = \frac{1}{8}(e_{123} + e_{126} - e_{135} + e_{156} + e_{234} - e_{246} + e_{345} + e_{456}).$
\end{enumerate}
Hence any spinor of the induced minimal left ideal $\psi\in \R_{3,3} f_{\textrm{U}(3)}$ is a linear combination of the $\R$ basis given above, and so we have
\[\R_{3,3} f_{\textrm{U}(3)}=\R\{f_{\textrm{U}(3)},e_1f_{\textrm{U}(3)},e_2f_{\textrm{U}(3)},e_{3}f_{\textrm{U}(3)},e_{12}f_{\textrm{U}(3)},e_{13}f_{\textrm{U}(3)},e_{23}f_{\textrm{U}(3)},e_{123}f_{\textrm{U}(3)}\}\cong \R^{8}.\] 

\subsection{Example for the Clifford algebra $\R_{3,4}$}
Just as we did with the previous example, we begin  with a  $\mathrm{U}(3)$-structure in $\R^{6}$ given by the Kahler form $\omega_{0}=e^{14}+e^{25}+e^{36}$.  From this  we have the same rational Kahler polynomial $P^{\mathbb{Q}}(\omega_{0})$ that induces  an embedded minimal left ideal in $\R_{3,4}$ associated to the primitive idempotent $f_{\mathrm{U}(3)}=\frac{1}{8}(1 + e_{14} + e_{25} + e_{36} - e_{1245} - e_{1346} - e_{2356} - e_{123456})$. 
The minimal left ideal  $\R_{3,4} f_{\mathrm{U(3)}}$, when viewed as a real vector space, is $16$-dimensional, and is generated by the following basis elements: 
\begin{enumerate}
    \item $f_{\mathrm{U}(3)}= \frac{1}{8}( - e_{123456} - e_{1245} - e_{1346} - e_{2356} + e_{14} + e_{25} + e_{36} + 1 )$
    \item $e_1f_{\mathrm{U}(3)} = \frac{1}{8}(- e_{12356} - e_{23456} + e_{125} + e_{136} - e_{245} - e_{346} + e_1 + e_4)$
    \item $e_2f_{\mathrm{U}(3)}= \frac{1}{8}(e_{12346} + e_{13456} - e_{124} + e_{145} + e_{236} - e_{356} + e_2 + e_5)$
    \item $e_3f_{\mathrm{U}(3)} = \frac{1}{8}(- e_{12345} - e_{12456} - e_{134} + e_{146} - e_{235} + e_{256} + e_3 + e_6)$
    \item $e_7f_{\mathrm{U}(3)} = \frac{1}{8}(- e_{1234567} - e_{12457} - e_{13467} - e_{23567} + e_{147} + e_{257} + e_{367} + e_7)$
    \item $e_{12}f_{\mathrm{U}(3)} = \frac{1}{8}(e_{1236} - e_{1356} + e_{2346} + e_{3456} + e_{12} + e_{15} - e_{24} + e_{45})$
    \item $e_{13}f_{\mathrm{U}(3)} = \frac{1}{8}(- e_{1235} + e_{1256} - e_{2345} - e_{2456} + e_{13} + e_{16} - e_{34} + e_{46})$
    \item $e_{17}f_{\mathrm{U}(3)} = \frac{1}{8}(- e_{123567} - e_{234567} + e_{1257} + e_{1367} - e_{2457} - e_{3467} + e_{17} + e_{47})$
    \item $e_{23}f_{\mathrm{U}(3)} = \frac{1}{8}(e_{1234} - e_{1246} + e_{1345} + e_{1456} + e_{23} + e_{26} - e_{35} + e_{56})$
    \item $e_{27}f_{\mathrm{U}(3)} = \frac{1}{8}(e_{123467} + e_{134567} - e_{1247} + e_{1457} + e_{2367} - e_{3567} + e_{27} + e_{57})$
    \item $e_{37}f_{\mathrm{U}(3)} = \frac{1}{8}(- e_{123457} - e_{124567} - e_{1347} + e_{1467} - e_{2357} + e_{2567} + e_{37} + e_{67})$
    \item $e_{123}f_{\mathrm{U}(3)} = \frac{1}{8}(e_{123} + e_{126} - e_{135} + e_{156} + e_{234} - e_{246} + e_{345} + e_{456})$
    \item $e_{127}f_{\mathrm{U}(3)} = \frac{1}{8}(e_{12367} - e_{13567} + e_{23467} + e_{34567} + e_{127} + e_{157} - e_{247} + e_{457})$
    \item $e_{137}f_{\mathrm{U}(3)} = \frac{1}{8}(- e_{12357} + e_{12567} - e_{23457} - e_{24567} + e_{137} + e_{167} - e_{347} + e_{467})$
    \item $e_{237}f_{\mathrm{U}(3)} = \frac{1}{8}(e_{12347} - e_{12467} + e_{13457} + e_{14567} + e_{237} + e_{267} - e_{357} + e_{567})$
    \item $e_{1237}f_{\mathrm{U}(3)} = \frac{1}{8}(e_{1237} + e_{1267} - e_{1357} + e_{1567} + e_{2347} - e_{2467} + e_{3457} + e_{4567})$
\end{enumerate}
That is, we have that as a real space, 
\begin{align*}
\R_{3,4} f_{\textrm{U}(3)} = \R & \{f_{\textrm{U}(3)},e_{1}f_{\textrm{U}(3)},e_2f_{\textrm{U}(3)},e_{3}f_{\textrm{U}(3)},e_{7}f_{\textrm{U}(3)},e_{12}f_{\textrm{U}(3)},e_{13}f_{\textrm{U}(3)},e_{17}f_{\textrm{U}(3)},e_{23}f_{\textrm{U}(3)}, \\
& e_{27}f_{\textrm{U}(3)},e_{37}f_{\textrm{U}(3)},e_{123}f_{\textrm{U}(3)},e_{127}f_{\textrm{U}(3)},e_{137}f_{\textrm{U}(3)},e_{237}f_{\textrm{U}(3)},e_{1237}f_{\textrm{U}(3)}\}\cong \R^{16}.
\end{align*}
The Clifford algebra $\R_{3,4}$ is of complex type. This is easily seen by the fact that when we consider $\R_{3,4}$ as a matrix algebra, we have the matrix algebra isomorphism $\R_{3,4}\cong \C(8)$. Thus the space of spinors must be isomorphic to $\C^{8}$ as a complex vector space, which implies that our minimal left ideal must have a complex structure. For the  Clifford algebra $\R_{3,4}$, a complex structure that can be chosen is the pseudoscalar (or volume element, as we refer to it) $\Gamma=e_{1234567}\in\R_{3,4}$. This pseudoscalar is central in the Clifford algebra, and also has the property of being negative definite; that is, $\Gamma^2=-1$, making $\Gamma$ a suitable choice for a complex structure on our minimal left ideal.
For this example, we make a sign change and identify the negative of the psuedoscalar as the complex structure. That is, we identify $(-\Gamma)$ with $i\in \C$, in order to preserve the orientation of our basis. With the aid of our complex structure, we define  complex multiplication by a complex scalar $a+bi\in\C$ on   our minimal left ideal as $(a+bi)\psi=a\psi+b(-\Gamma)\psi$, for any $\psi\in\R_{3,4} f_{\textrm{U}(3)}$.
With complex multiplication defined in this manner, we make $\R_{3,4} f_{\textrm{U}(3)}$ into a left $\C$-module, where we view the complex numbers as the  subalgebra $\R\{1,-\Gamma\}\cong \C$. Since $-\Gamma$ commutes with the primitive idempotent $f_{\textrm{U(3)}}$, we have $(f_{\textrm{U}(3)} (-\Gamma) f_{\textrm{U}(3)})^2=\Gamma^2f_{\textrm{U}(3)}^2=-f_{\textrm{U}(3)}$. Then we can identify the complex scalar $i$ with the term  $f_{\textrm{U}(3)} (-\Gamma) f_{\textrm{U}(3)}$, when we view  $\R_{3,4} f_{\textrm{U}(3)}$ as a right $\mathfrak{O}$-module isomorphic to $\C$.
It then follows that for the minimal left ideal $\R_{3,4} f_{\textrm{U}(3)}$, the negative pseudoscalar defines both left and right multiplication by $i$, and hence it is also used to define $\R_{3,4} f_{\textrm{U}(3)}$ as a complex vector space, where 
we define the following basis elements in terms of the complex generator $-
\Gamma$.
\begin{enumerate}
    \item $-\Gamma f_{\textrm{U}(3)}=e_{7} f_{\textrm{U}(3)}$,
    \item $e_{1}(-\Gamma) f_{\textrm{U}(3)}=e_{17} f_{\textrm{U}(3)}$,
    \item $e_{2}(-\Gamma )f_{\textrm{U}(3)}=e_{27} f_{\textrm{U}(3)}$,
\item $e_{3}(-\Gamma )f_{\textrm{U}(3)}=e_{37} f_{\textrm{U}(3)}$,
\item $e_{12}(-\Gamma )f_{\textrm{U}(3)}=e_{127} f_{\textrm{U}(3)}$,
\item $e_{13}(-\Gamma )f_{\textrm{U}(3)}=e_{137} f_{\textrm{U}(3)}$,
\item $e_{23}(-\Gamma )f_{\textrm{U}(3)}=e_{237} f_{\textrm{U}(3)}$,
\item $e_{123}(-\Gamma )f_{\textrm{U}(3)}=e_{1237} f_{\textrm{U}(3)}$.
\end{enumerate}

With this identification, it is easy to see that we can break up 
$\R_{3,4} f_{\textrm{U}(3)}$ into the direct sum $\R_{3,3} f_{\textrm{U}(3)}\oplus (-\Gamma)\R_{3,3}  f_{\textrm{U}(3)}$, where we define  $\R_{3,3} f_{\textrm{U}(3)}$ as 
\[
\R_{3,3} f_{\textrm{U}(3)}=\R\{f_{\textrm{U}(3)},e_1f_{\textrm{U}(3)},e_2f_{\textrm{U}(3)},e_{3}f_{\textrm{U}(3)},e_{12}f_{\textrm{U}(3)},e_{13}f_{\textrm{U}(3)},e_{23}f_{\textrm{U}(3)},e_{123}f_{\textrm{U}(3)}\}.
\]
With this identification, we can view our minimal left ideal as an $8$-dimensional complex vector space given by 
\[
\R_{3,4} f_{\textrm{U}(3)}=\C\{f_{\textrm{U}(3)},e_{1}f_{\textrm{U}(3)},e_{2}f_{\textrm{U}(3)},e_{3}f_{\textrm{U}(3)},e_{12}f_{\textrm{U}(3)},e_{13}f_{\textrm{U}(3)},e_{23}f_{\textrm{U}(3)},e_{123}f_{\textrm{U}(3)}\}\cong \C^{8}.
\]

\subsection{Example for the Clifford algebra $\R_{3,5}$}

For the Clifford algebra $\R_{3,5}$, which is of quaternionic type since we have the matrix algebra isomorphism  $\R_{3,5}\cong \mathbb{H}(8)$, we must define a quaternionic structure where the generators that define a right $\mathbb{H}$ basis on the minimal left ideal also establish an isomorphism  $\mathfrak{O}\cong \mathbb{H}$. As in the previous two examples,  we have that $f_{\textrm{U}(3)}$ is a primitive idempotent, this time for the Clifford algebra $\R_{3,5}$.
Now the minimal left ideal is of real dimension $32$; that is, $\R_{3,5} f_{\textrm{U}(3)}\cong\R^{32}$, with the following $\R$ basis:
\begin{enumerate}
    \item $ f_{\textrm{U}(3)} = \frac{1}{8}( - e_{123456} - e_{1245} - e_{1346} - e_{2356} + e_{14} + e_{25} + e_{36} + 1 )$
    \item $e_1 f_{\textrm{U}(3)} = \frac{1}{8}( - e_{12356} - e_{23456} + e_{125} + e_{136} - e_{245} - e_{346} + e_1 + e_4)$
    \item $e_2 f_{\textrm{U}(3)} = \frac{1}{8}(e_{12346} + e_{13456} - e_{124} + e_{145} + e_{236} - e_{356} + e_2 + e_5)$
    \item $e_3 f_{\textrm{U}(3)} = \frac{1}{8}(- e_{12345} - e_{12456} - e_{134} + e_{146} - e_{235} + e_{256} + e_3 + e_6)$
    \item $e_{12} f_{\textrm{U}(3)} = \frac{1}{8}(e_{1236} - e_{1356} + e_{2346} + e_{3456} + e_{12} + e_{15} - e_{24} + e_{45})$
    \item $e_{13} f_{\textrm{U}(3)} = \frac{1}{8}(- e_{1235} + e_{1256} - e_{2345} - e_{2456} + e_{13} + e_{16} - e_{34} + e_{46})$
    \item $e_{23} f_{\textrm{U}(3)} = \frac{1}{8}(e_{1234} - e_{1246} + e_{1345} + e_{1456} + e_{23} + e_{26} - e_{35} + e_{56})$
    \item $e_{123} f_{\textrm{U}(3)} = \frac{1}{8}(e_{123} + e_{126} - e_{135} + e_{156} + e_{234} - e_{246} + e_{345} + e_{456})$
    \item $e_7 f_{\textrm{U}(3)} = \frac{1}{8}(- e_{1234567} - e_{12457} - e_{13467} - e_{23567} + e_{147} + e_{257} + e_{367} + e_7)$
    \item $e_{17} f_{\textrm{U}(3)}= \frac{1}{8}(- e_{123567} - e_{234567} + e_{1257} + e_{1367} - e_{2457} - e_{3467} + e_{17} + e_{47})$
    \item $e_{27} f_{\textrm{U}(3)} = \frac{1}{8}(e_{123467} + e_{134567} - e_{1247} + e_{1457} + e_{2367} - e_{3567} + e_{27} + e_{57})$
    \item $e_{37} f_{\textrm{U}(3)} = \frac{1}{8}(- e_{123457} - e_{124567} - e_{1347} + e_{1467} - e_{2357} + e_{2567} + e_{37} + e_{67})$
    \item $e_{127} f_{\textrm{U}(3)} = \frac{1}{8}(e_{12367} - e_{13567} + e_{23467} + e_{34567} + e_{127} + e_{157} - e_{247} + e_{457})$
    \item $e_{137} f_{\textrm{U}(3)} = \frac{1}{8}(- e_{12357} + e_{12567} - e_{23457} - e_{24567} + e_{137} + e_{167} - e_{347} + e_{467})$
    \item $e_{237} f_{\textrm{U}(3)}= \frac{1}{8}(e_{12347} - e_{12467} + e_{13457} + e_{14567} + e_{237} + e_{267} - e_{357} + e_{567})$
    \item $e_{1237} f_{\textrm{U}(3)} = \frac{1}{8}(e_{1237} + e_{1267} - e_{1357} + e_{1567} + e_{2347} - e_{2467} + e_{3457} + e_{4567})$
    \item $e_8 f_{\textrm{U}(3)} = \frac{1}{8}(- e_{1234568} - e_{12458} - e_{13468} - e_{23568} + e_{148} + e_{258} + e_{368} + e_8)$
    \item $e_{18} f_{\textrm{U}(3)} = \frac{1}{8}(- e_{123568} - e_{234568} + e_{1258} + e_{1368} - e_{2458} - e_{3468} + e_{18} + e_{48})$
    \item $e_{28} f_{\textrm{U}(3)} = \frac{1}{8}(e_{123468} + e_{134568} - e_{1248} + e_{1458} + e_{2368} - e_{3568} + e_{28} + e_{58})$
    \item $e_{38} f_{\textrm{U}(3)} = \frac{1}{8}(- e_{123458} - e_{124568} - e_{1348} + e_{1468} - e_{2358} + e_{2568} + e_{38} + e_{68})$
    \item $e_{128} f_{\textrm{U}(3)} = \frac{1}{8}(e_{12368} - e_{13568} + e_{23468} + e_{34568} + e_{128} + e_{158} - e_{248} + e_{458})$
    \item $e_{138} f_{\textrm{U}(3)} = \frac{1}{8}(- e_{12358} + e_{12568} - e_{23458} - e_{24568} + e_{138} + e_{168} - e_{348} + e_{468})$
    \item $e_{238} f_{\textrm{U}(3)} = \frac{1}{8}(e_{12348} - e_{12468} + e_{13458} + e_{14568} + e_{238} + e_{268} - e_{358} + e_{568})$
    \item $e_{1238} f_{\textrm{U}(3)} = \frac{1}{8}(e_{1238} + e_{1268} - e_{1358} + e_{1568} + e_{2348} - e_{2468} + e_{3458} + e_{4568})$
    \item $e_{78} f_{\textrm{U}(3)} = \frac{1}{8}(- e_{12345678} - e_{124578} - e_{134678} - e_{235678} + e_{1478} + e_{2578} + e_{3678} + e_{78})$
    \item $e_{178} f_{\textrm{U}(3)} = \frac{1}{8}(- e_{1235678} - e_{2345678} + e_{12578} + e_{13678} - e_{24578} - e_{34678} + e_{178} + e_{478})$
    \item $e_{278} f_{\textrm{U}(3)} = \frac{1}{8}(e_{1234678} + e_{1345678} - e_{12478} + e_{14578} + e_{23678} - e_{35678} + e_{278} + e_{578})$
    \item $e_{378} f_{\textrm{U}(3)} = \frac{1}{8}(- e_{1234578} - e_{1245678} - e_{13478} + e_{14678} - e_{23578} + e_{25678} + e_{378} + e_{678})$
    \item $e_{1278} f_{\textrm{U}(3)}= \frac{1}{8}(e_{123678} - e_{135678} + e_{234678} + e_{345678} + e_{1278} + e_{1578} - e_{2478} + e_{4578})$
    \item $e_{1378} f_{\textrm{U}(3)}= \frac{1}{8}(- e_{123578} + e_{125678} - e_{234578} - e_{245678} + e_{1378} + e_{1678} - e_{3478} + e_{4678})$
    \item $e_{2378} f_{\textrm{U}(3)} = \frac{1}{8}(e_{123478} - e_{124678} + e_{134578} + e_{145678} + e_{2378} + e_{2678} - e_{3578} + e_{5678})$
    \item $e_{12378} f_{\textrm{U}(3)} = \frac{1}{8}(e_{12378} + e_{12678} - e_{13578} + e_{15678} + e_{23478} - e_{24678} + e_{34578} + e_{45678})$
\end{enumerate}
As we have stated, the Clifford algebra $\R_{3,5}$ is of quaternionic type, and hence has a quaternionic subalgebra, where the left actions of this subalgebra can be identified with left actions by $\mathbb{H}$ on our minimal left ideal $\R_{3,5} f_{\textrm{U}(3)}$.  The Clifford algebra $\R_{3,5}$ has a subalgebra isomorphic to $\mathbb{H}$ established by the generators $1,e_{7},e_{8}$ and $e_{78}$, where we set $e_{7}\cong i, e_{8}\cong j$, and $e_{78}\cong k$. Thus we have $\R\{1,e_{7},e_{8},e_{78}\}\cong \mathbb{H}$.

These generators also have the property that they commute with the primitive idempotent induced from our $\textrm{U}(3)$-structure. That is, $e_{7}f_{\textrm{U}(3)}=f_{\textrm{U}(3)}e_{7}$, $e_{8}f_{\textrm{U}(3)}=f_{\textrm{U}(3)}e_{8}$, and $e_{78}f_{\textrm{U}(3)}=f_{\textrm{U}(3)}e_{78}.$

Moreover, for all three generators $e_{7}, e_{8},$ and $e_{78}$, we have in $\mathfrak{O}=f_{\textrm{U}(3)} \R_{3,5} f_{\textrm{U}(3)}$ the property $(f_{\textrm{U}(3)}e_{7}f_{\textrm{U}(3)})^2=(f_{\textrm{U}(3)}e_{8}f_{\textrm{U}(3)})^2=(f_{\textrm{U}(3)}e_{78}f_{\textrm{U}(3)})^2=-f_{\textrm{U}(3)}$, as well as the following relations:
\begin{enumerate}[label=(\roman*)]
    \item $f_{\textrm{U}(3)}fe_{7}f_{\textrm{U}(3)}\cdot f_{\textrm{U}(3)}e_{8}f_{\textrm{U}(3)}=f_{\textrm{U}(3)}e_{78}f_{\textrm{U}(3)}$
    \item $f_{\textrm{U}(3)}e_{8}f_{\textrm{U}(3)}\cdot f_{\textrm{U}(3)}e_{78}f_{\textrm{U}(3)}=f_{\textrm{U}(3)}e_{7}f_{\textrm{U}(3)}$
  \item $f_{\textrm{U}(3)}e_{78}f_{\textrm{U}(3)}\cdot f_{\textrm{U}(3)}e_{7}f_{\textrm{U}(3)}=f_{\textrm{U}(3)}e_{8}f_{\textrm{U}(3)}$
  \item $f_{\textrm{U}(3)}e_{8}f_{\textrm{U}(3)}\cdot f_{\textrm{U}(3)}e_{7}f_{\textrm{U}(3)}=-f_{\textrm{U}(3)}e_{78}f_{\textrm{U}(3)}$
    \item $f_{\textrm{U}(3)}e_{78}f_{\textrm{U}(3)}\cdot f_{\textrm{U}(3)}e_{8}f_{\textrm{U}(3)}=-f_{\textrm{U}(3)}e_{7}f_{\textrm{U}(3)}$
  \item $f_{\textrm{U}(3)}e_{7}f_{\textrm{U}(3)}\cdot f_{\textrm{U}(3)}e_{78}f_{\textrm{U}(3)}=-f_{\textrm{U}(3)}e_{8}f_{\textrm{U}(3)}$
    \end{enumerate}

Under the identification $f_{\textrm{U}(3)}e_{7}f_{\textrm{U}(3)}\cong i$,  $f_{\textrm{U}(3)}e_{8}f_{\textrm{U}(3)}\cong j$, and $f_{\textrm{U}(3)}e_{78}f_{\textrm{U}(3)}\cong k$, we can establish that $\mathfrak{O}=f_{\textrm{U}(3)}\R_{3,5}f_{\textrm{U}(3)}\cong \mathbb{H}$, making  the minimal left ideal $\R_{3,5}f_{\textrm{U}(3)}$ into a right $\mathbb{H}$ module.
To see how right multiplication by the generators of $\mathbb{H}$  acts on $\R_{3,5}f_{\textrm{U}(3)}$, we have the following, for an arbitrary element  $g f_{\textrm{U}(3)} \in\R_{3,5}f_{\textrm{U}(3)}$: 
\begin{enumerate}
    \item $(g f_{\textrm{U}(3)}) \cdot i\mapsto (g f_{\textrm{U}(3)}) (f_{\textrm{U}(3)}e_{7}f_{\textrm{U}(3)})=gf_{\textrm{U}(3)}^{2}e_{7} f_{\textrm{U}(3)}=gf_{\textrm{U}(3)}e_{7}f_{\textrm{U}(3)}=g e_{7} f_{\textrm{U}(3)}^2=g e_{7}f_{\textrm{U}(3)}.$
   \item $(g f_{\textrm{U}(3)}) \cdot j\mapsto (g f_{\textrm{U}(3)}) (f_{\textrm{U}(3)}e_{8}f_{\textrm{U}(3)})=gf_{\textrm{U}(3)}^{2}e_{8} f_{\textrm{U}(3)}=gf_{\textrm{U}(3)}e_{8}f_{\textrm{U}(3)}=g e_{8} f_{\textrm{U}(3)}^2=g e_{8}f_{\textrm{U}(3)}.$
   \item $(g f_{\textrm{U}(3)}) \cdot k\mapsto (g f_{\textrm{U}(3)}) (f_{\textrm{U}(3)}e_{78}f_{\textrm{U}(3)})=gf_{\textrm{U}(3)}^{2}e_{78} f_{\textrm{U}(3)}=gf_{\textrm{U}(3)}e_{78}f_{\textrm{U}(3)}=g e_{78} f_{\textrm{U}(3)}^2=g e_{78}f_{\textrm{U}(3)}.$
\end{enumerate}
Hence the generators $e_{7},e_{8}, $ and $e_{78}$  define  right and left multiplication by $\mathbb{H}$, as well as the right $\mathbb{H}$ basis for the minimal left ideal $\R_{3,5}f_{\textrm{U}(3)}$ when viewed as a right $\mathbb{H}$ module. With this identification, it is easy to see that we can break up $\R_{3,5} f_{\textrm{U}(3)}$ into the direct sum 
\begin{align*}
\R_{3,5} f_{\textrm{U}(3)} & =\R_{3,3} f_{\textrm{U}(3)}\oplus \R_{3,3}  f_{\textrm{U}(3)} ( f_{\textrm{U}(3)} e_{7}  f_{\textrm{U}(3)}) \\
& \quad \quad \oplus \R_{3,3}  f_{\textrm{U}(3)} ( f_{\textrm{U}(3)} e_{8}  f_{\textrm{U}(3)}) 
\oplus \R_{3,3}  f_{\textrm{U}(3)} ( f_{\textrm{U}(3)} e_{78}  f_{\textrm{U}(3)}) \\
& \cong \R_{3,3} f_{\textrm{U}(3)}\oplus \R_{3,3}  f_{\textrm{U}(3)} i\oplus \R_{3,3} j \oplus \R_{3,3}  f_{\textrm{U}(3)} k,
\end{align*}
where $\R_{3,3} f_{\textrm{U}(3)}=\R\{f_{\textrm{U}(3)},e_1f_{\textrm{U}(3)},e_2f_{\textrm{U}(3)},e_{3}f_{\textrm{U}(3)},e_{12}f_{\textrm{U}(3)},e_{13}f_{\textrm{U}(3)},e_{23}f_{\textrm{U}(3)},e_{123}f_{\textrm{U}(3)}\} $. 

Our minimal left ideal $\R_{3,5} f_{\textrm{U}(3)}$ also has the structure of a left Clifford module, as all minimal left ideals do,  where multiplication by $\mathbb{H}$ on the left is identified with the quaternionic subalgebra $\R\{1,e_{7},e_{8},e_{78}\}\subset \R_{3,5}$. Our minimal left ideal being a left and right $\mathbb{H}$-module gives us  
\begin{align*}
\R_{3,5}f_{\textrm{U}(3)} =\mathbb{H}\{f_{\textrm{U}(3)},e_{1}f_{\textrm{U}(3)},e_{2}f_{\textrm{U}(3)},e_{3}f_{\textrm{U}(3)},e_{12}f_{\textrm{U}(3)},e_{13}f_{\textrm{U}(3)},e_{23}f_{\textrm{U}(3)},e_{123}f_{\textrm{U}(3)}\} \cong \mathbb{H}^{8},
\end{align*}
but unlike in $\R_{3,4}f_{\textrm{U}(3)}$, the quaternion elements do not commute with  basis elements of the minimal left ideal.
\subsubsection{Conclusion about $\textrm{U}(3)$-structures and their associated minimal left ideals}
As we saw above, the $\textrm{U}(3)$-structure in $\R^{6}$ defined as the stabilizer of the Kahler form $\omega_{0}=e^{12}+e^{34}+e^{56}$ and the standard inner product $g_{0}=(e^{1})^2+\cdots + (e^{6})^2$ have an induced minimal left ideal $\R_{3,3} f_{\textrm{U}(3)}$, where the primitive idempotent  is given by $q^{*}(P^{\mathbb{Q}}(\omega_{0}))=f_{\textrm{U}(3)}$. 
One can then view the complex structure and the quaternionic structures in the minimal left ideals  
$\R_{3,4} f_{\textrm{U}(3)}$ and $\R_{3,5} f_{\textrm{U}(3)}$ as direct sums of copies of the minimal left ideal  $\R_{3,3} f_{\textrm{U}(3)}$, showing that the induced embedded minimal left ideals for the Clifford algebras $\R_{3,4}$ and $\R_{3,5}$ are all induced and well understood from the same $\textrm{U}(3)$-structure in $\R^{6}.$

\section{$\textrm{U}(n)$-structures recovered from Clifford algebras of certain signatures}

In this section, we start with any Clifford algebra $\R_{p,q}$ and ask, for which signatures can we recover a $\textrm{U}(m)$-structure, and for which signatures can we recover a $\textrm{U}(m)$-structure by projection? For the $\textrm{U}(m)$-structures recoverable by projection, we then ask how the image of the primitive idempotent looks in the larger exterior algebra $\bigwedge^{\bullet}(\R^{p+q})^{*}$. We begin with an immediate corollary to Theorem \ref{Kahler}, in which we identify the recovery of a $\textrm{U}(p)$-structure on $\R^{2p}$, directly or from projection, from Clifford algebras of signatures $(p,p)$, $(p,p+1),$ and $(p,p+2)$.

\begin{cor}\label{Kahler Corrollary}
For  real Clifford algebras $\R_{p,p}$,  we can always recover a $\textrm{U}(p)$-structure  on $\R^{2p}$  from  minimal left ideals  $\R_{p,p}f$; and for the real Clifford algebras  $\R_{p,p+1}$ and $\R_{p,p+2}$, we can recover a $\textrm{U}(p)$-structure  on $\R^{2p}$ by projection, where the rational Kahler polynomial on $\R^{2p}$ is given by $\sigma^{*}(f)=P^{\mathbb{Q}}(\omega_{0})$, and  the Kahler form is given by the expression $\sigma^{*}(\langle2^{m} f\rangle_{2})=\omega_{0}$.
    
\end{cor}

\begin{proof}

For the real Clifford algebras $\R_{p,p}$, $\R_{p,p+1},$ and $\R_{p,p+2}$ such that  $p\geq 1$, the primitive idempotents $f$ that generate the minimal left ideals are generated by $p$ commuting involutions. From Theorem \ref{Kahler}, we have that the minimal left ideal for each of these Clifford algebras is given by $f_{\textrm{U}(p)}:=q^{*}(P^{\mathbb{Q}}(\omega_{0}))=\prod_{i=1}^{p}\big(\dfrac{1+e_{i} e_{i+p}}{2}\big )$, where $\omega_{0}$ is the Kahler form on $\R^{2p}$ induced by the standard Euclidean inner product and the complex structure given by $J(e_{k})=e_{k+p}$ and $J(e_{k+p})=-e_{p}$ for $k=1,2,\ldots,p$. Hence if we define $f_p:=2^{p} f_{\textrm{U}(p)}:=2^{p}q^{*}(P^{\mathbb{Q}}(\omega_{0}))$, then the image of $f_{p}$ under the symbol map $\sigma^{*}:\R_{p,p}\rightarrow \bigwedge^{\bullet}(\R^{2p})^{*}$ is the Kahler polynomial $P(\omega_{0})$; that is,   $P(\omega_{0})=\sigma^{*}(f_{p})$.  From this, if we restrict $f_{p}$ to its degree $2$ components, we recover the Kahler form in $\R^{2p}$; that is, $\sigma^{*}(\langle f_{p}\rangle_{2})=\omega_{0}$. Thus, for the Clifford algebra $\R_{p,p}$, we are able to recover a $\textrm{U}(p)$-structure on $\R^{2p}$. For the Clifford algebras $\R_{p,p+1}$ and $\R_{p,p+2}$, we define the subalgebra generated by the standard basis vectors in $\R^{2p}$ as $A=\R\langle e_1,\ldots,e_{2p}\rangle$. It then follows that 
 $A\cong \R_{p,p}$. The image of this algebra under the extended symbol map  is isomorphic to $\bigwedge^{\bullet}(\R^{2p})^{*}$. Hence, for the Clifford algebras $\R_{p,p+1}$ and $\R_{p,p+2}$, we have $\textrm{U}(p)$-structures recovered by projection on $\R^{2p}$ via $\sigma^{*}\big|_{A}(\langle f_{p}\rangle_{2})=\tilde{\omega}_{0}$, where we define $\tilde{\omega}_{0}$ as the Kahler form in the underlying vector space that generates $\sigma^{*}(A)$ isomorphic to $\R^{2p}$. The Kahler form $\tilde{\omega}_{0}$ is isomorphic to $\omega_{0}$ in $\R^{2p}$ when the underlying vector spaces are identified. Since $f_{p}$ is an algebraic expression of $f_{\textrm{U}(p)}$ in the Clifford algebras $\R_{p,p+1}$ and $\R_{p,p+2}$, we have that from minimal left ideals on the Clifford algebras $\R_{p,p+1}$ and $\R_{p,p+2}$, we are able to recover by projection a $\textrm{U}(p)$-structure on $\R^{2p}$.

\end{proof}
\begin{remark}
In the proof of Corollary \ref{Kahler Corrollary}, we do not make any mention of recovering the almost complex structure $J_{0}$ on $\R^{2p}$ needed to have a $\textrm{U}(p)$-structure. The reason is that $\R^{2p}$ has the standard Euclidean inner product, and since we can recover the Kahler form $\omega_{0}$ from $f_{\textrm{U}(p)}$, it then follows that we can easily recover $J_{0}$ on $\R^{2p}$. This is because in a $\textrm{U}(p)$-structure, any two parallel tensors define the third. In fact, given that $\sigma^{*}(\langle f \rangle_2)=\omega_{0}=e^{1}\wedge e^{p+1}+\cdots+e^{p}\wedge e^{2p}$, it is clear that the recovered complex structure $J_{0}$ is defined by $J(e_{k})=e_{k+m}$ and $J(e_{k+m})=-e_{m}$ for $k=1,2,\ldots,p$. 
\end{remark}

We now turn our attention to Clifford algebras of signature $(p,q)$ in general, for $p \neq q$ and $p, q \neq 0$. In the general case, we do not recover a $\textrm{U}(m)$-structure projectively from the primitive idempotent that defines the minimal left ideal, as was the case for $\R_{p,p}$, $\R_{p,p+1}$, and $\R_{p,p+2}$. Instead, the $\textrm{U(p)}$-structure is projectively recovered from a factor of the primitive idempotent that defines a minimal left ideal for Clifford algebras of signature $(p,q)$.

\begin{prop}\label{Recovery theorem}
For real Clifford algebras  of signature $(p,q)$ such that $p \neq q$, $p, q \neq 0$, and $q \not\in \{p+1, p+2\}$, there exists a  factor $\tilde{f}$ of a primitive idempotent $h$ that  recovers by projection a $\textrm{U}(m)$ on  $\R^{2m}$, where $m= \textrm{min}\{p,q\}$. Moreover, the image of the primitive idempotent $h=\tilde{f} e$ in  $\bigwedge^{\bullet}(\R^{p+q})^{*}$ via the extended symbol map splits as the product  $\sigma^{*}(h)=P^{\mathbb{Q}}(\tilde{\omega_{0}})\wedge\sigma^{*}(e)$, %where $\sigma^{*}(e)\in\bigwedge^{\bullet}(\R^{l})^{*}$ such that  $l=p+q-2m$. 
where $\tilde{\omega}_{0}$ is a $2$-form in $\R^{p+q}$ isomorphic to the Kahler form in $\R^{2m}$ that defines the $\textrm{U}(m)$-structure projectively recovered in $\R^{2m}$. %Moreover in  $\sigma^{*}(\R_{p,q})=\bigwedge^{\bullet}(\R^{p+q})^{*}$ we can recover a differential $2$-form given by the expression  $\sigma^{*}(2^{k}\langle \tilde{f}\cdot e\rangle_{2})=\tilde{\omega}\wedge e_{2},$ where $e_{2}$ is two from that  acts on $\R^{p+q}\setminus \R^{2m}$ if $e$ has degree two monomials in its expanded form, otherwise $e_{2}=1$.
\end{prop}

\begin{proof}
For Clifford algebras $\R_{p,q}$  such that $p \neq q$, $p, q \neq 0$, and $q \not\in \{p+1, p+2\}$, we break down our analysis into two cases.
 \begin{enumerate}
 \item If $p>q$, then we can write $p=q+l$ for some $l>0$. Thus, for the Clifford algebras $\R_{q+l,q}$,  we can index $q$ generators of degree $2$ of the canonical basis as   $e_{\sigma_{j}}=e_{j} e_{_{q+j}}$  for $j=1,\ldots,q$.  We then define $\tilde{f}=(1+e_{\sigma_1})\cdots(1+e_{\sigma_{q}})$, where $e_{\sigma_{j}}^2=1$ and each $e_{\sigma_{j}}$ commutes with all other  elements $e_{\sigma_{k}}$ where $j\neq k$.  The  expanded form $\tilde{f}$ is given by   
\[
\tilde{f}=1+\sum_{k=1}^{q} \, \sum_{1\leq j_{1}<\cdots <j_{k}\leq q}e_{\sigma_{j_{1}}}\cdots e_{\sigma_{j_{k}}}.
\]
 
In the Clifford algebra $\R_{q+l,q}$, we denote the subalgebra generated by the  basis vectors $e_{1},\ldots,e_{q},e_{q+1+l},\ldots,e_{2q+l}$ as  $A=\R\langle e_{1},\ldots,e_{q},e_{q+1+l},\ldots,e_{2q+l}\rangle$.  It is immediate that this Clifford algebra $A$ is isomorphic to the real Clifford algebra $\R_{q,q}$.
Restricting the symbol map to the subalgebra $A$, we have at the level of images that $\sigma^{*}\big|_{A}(A)\cong \bigwedge^{\bullet}(\R^{2q})^{*}$. Utilizing this isomorphism, we now have $\sigma^{*}\big|_{A}(\tilde{f})=P(\omega_{0})$, where $\sigma\big|_{A}^{*}(\langle \tilde{f}\rangle_2)=\sum_{j=1}^{q}\sigma^{*}\big|_{A}(e_{\sigma_{j}})=\sum_{j=1}^{q}e^{j}\wedge e^{q+j}:=\omega_{0}\in \wedge^{2}(\R^{2q})^{*}$. Hence $\sigma^{*}\big|_{A}(\tilde{f})$ gives us our Kahler form $\omega_{0}$ on $\R^{2q}$, and thus we have a $\textrm{U}(q)$-structure recovered by the projection of the element $\tilde{f}\in A$. For the minimal left ideal associated with the Clifford algebra $\R_{q+l,q}$, the primitive idempotent that defines it, $h$, is the product of $k$ commuting involutions (see Theorem  \ref{Minimal left ideals}). Considering that in the construction of $\tilde{f}$ we already have $q$ commuting positive definite generators, we may  partition $k$ as $k=q+m$ for some $m\geq 0$. Thus, to generate a primitive idempotent $h$ in $\R_{q+l,q}$, we need an additional $m$ more commuting positive definite generators chosen from the subalgebra $\R\langle e_{q+1},\ldots,e_{q+l}\rangle \cong \R_{l,0}$; we denote these generators by  $e_{\alpha_1},\dots,e_{\alpha_{m}}$. Hence our primitive idempotent will be of the form $h=(\dfrac{1}{2^{q}}\tilde{f})(e)=\dfrac{1}{2^{q}}(\tilde{f} e)$, where $e=\dfrac{1}{2^{m}}(1+e_{\alpha_{1}})\cdots (1+e_{\alpha_{m}}).$ By construction, this product is a sum of elementary tensors. Therefore, the symbol map preserves the Clifford product that makes up $h$. That is, $\sigma^{*}(\dfrac{1}{2^{q}}\tilde{f}   e)=\sigma^{*}(\dfrac{1}{2^{q}}\tilde{f})\wedge \sigma^{*}( e)=P^{\mathbb{Q}}(\tilde{\omega_{0}})\wedge \sigma^{*}(e)$, where $\tilde{\omega}_{0}$ is the image of the  $2$-form in $\R^{2q+l}$ generated by $\tilde{f}\in A$ isomorphic to the Kahler $2$-form $\omega_{0}$ in $\R^{2q}$ that defines our $\textrm{U}(q)$-structure recovered by projection, and $\sigma^{*}(e)$ is an exterior form in the subalgebra $\bigwedge^{\bullet}(\R^{l})^{*}\subset \bigwedge^{\bullet}(\R^{2q+l})^{*}$, where we view $\R^{l}=\R\{e_{q+1},\ldots,e_{q+l}\}$, and we view $\R^{2q+l}=\R^{2q}\oplus\R^{l}$.

\item If $p<q$, then we can write $q=p+l$ for some $l>0$. Then on the Clifford algebras $\R_{p,p+l}$,  we can index as we did above: $e_{\sigma_{j}}=e_{j} e_{_{p+j}}$  for $j=1,\ldots,p$, where $e_{\sigma_{j}}^2=1$ for $j=1, \ldots, p$, and the  $e_{\sigma_{j}}$ are commuting elements in $\R_{p,p+l}$. When we expand the product, we obtain
\[
\tilde{f}=1+\sum_{k=1}^{p}\sum_{1\leq j_{1}<j_{2}<\cdots <j_{k} \leq p}e_{\sigma_{j_{1}}}\cdots e_{\sigma_{j_{k}}}.
\]
In the Clifford algebra $\R_{p,p+l}$, we have that the subalgebra $A=\R\langle e_{1},\ldots,e_{2p} \rangle$ is isomorphic to the Clifford algebra $\R_{p,p}$.
Again, when we restrict the symbol map to the subalgebra $A$, we have at the level of images that $\sigma^{*}\big|_{A}(A)\cong \bigwedge^{\bullet}(\R^{2p})^{*}$. Using this isomorphism,  we now have $\sigma^{*}\big|_{A}(\tilde{f})=P(\omega_{0})$, where $\sigma\big|_{A}^{*}(\langle \tilde{f}\rangle_2)=\sum_{j=1}^{p}\sigma^{*}\big|_{A}(e_{\sigma_{j}})=\sum_{j=1}^{p}e^{j}\wedge e^{p+j}:=\omega_{0}\in \wedge^{2}(\R^{2p})^{*}$. Hence we have a $\textrm{U}(p)$-structure $\R^{2p}$ recovered by the projection of our constructed element $\tilde{f}\in A$. 
Since in the construction of $\tilde{f}$ we already have $p$ commuting positive definite generators, we may  partition $k$ as $k=p+m$ for some $m\geq 0$ (where $m=0$ when we have signature $(p,p+1)$ or $(p,p+2)$).
Since we have $p$ commuting involutions used for the construction of the element $\tilde{f}$, where $\sigma^{*}(\tilde{f})=P(\omega_{0})$, the primitive idempotent that generates the minimal left ideal can be generated using these $p$ commuting involutions and an additional $m$ commuting involutions chosen from the subalgebra $\R\langle e_{2p+1},\ldots,e_{2p+l}\rangle\cong\R_{0,l}$, which again we denote $e_{\alpha_1},\ldots,e_{\alpha_{m}}$.
Our primitive idempotent will be of the form $h=(\dfrac{1}{2^{p}}\tilde{f})(e)=\dfrac{1}{2^{p}}(\tilde{f} e)$, where $e=\dfrac{1}{2^{m}}(1+e_{\alpha_{1}})\cdots (1+e_{\alpha_{m}}).$ 
  Given that this product is the sum of monomials of elementary tensors (basis elements in the Clifford algebra $\R_{p,p+l}$) that do not overlap in indices, the symbol map preserves the product. Thus, we have  $\sigma^{*}(\dfrac{1}{2^{p}}\tilde{f}   e)=\sigma^{*}(\dfrac{1}{2^{p}}\tilde{f})\wedge \sigma^{*}(  e)=P^{\mathbb{Q}}(\tilde{\omega_{0}})\wedge \sigma^{*}(e)$, where $\tilde{\omega}_{0}$ is the image of the  $2$-form in $\R^{2p+l}$ generated by $\tilde{f}\in A$ isomorphic to the Kahler $2$-form $\omega_{0}$ in $\R^{2p}$ that defines our $\textrm{U}(q)$-structure recovered by projection, and $\sigma^{*}(e)$ is an exterior form in the subalgebra $\bigwedge^{\bullet}(\R^{l})^{*}\subset \bigwedge^{\bullet}(\R^{2p+l})^{*}$, where we view $\R^{l}=\R\{e_{2p+1},\ldots,e_{2p+l}\}$, and we view $\R^{2p+l}=\R^{2p}\oplus\R^{l}$.
   \end{enumerate}
\vspace{2mm}

Hence we have shown that for any signature $(p,q)$ where $p \neq q$, $p, q \neq 0$, and $q \not\in \{p+1, p+2\}$, we can find a factor of a minimal left ideal that recovers a $\textrm{U}(m)$-structure through projection. Moreover, the image of the primitive idempotent $h$ in the exterior algebra $\bigwedge^{\bullet}(\R^{p+q})^{*}$ is the product of the rational Kahler polynomial $P^{\mathbb{Q}}(\tilde{\omega}_{0})$, where $\tilde{\omega}_{0}$ is a $2$-form in $\R^{p+q}$ isomorphic to the Kahler form in $\R^{2m}$, and an exterior form $\sigma^{*}(e)$ in the exterior subalgebra $\bigwedge^{\bullet}(\R^{l})^{*}$.

%Thus, we conclude, that for all Clifford algebras $\R_{p,q}$, where $p\neq q\neq 0$ , an element $\tilde{f}$ such that $\sigma\big|_{A}^{*}(f)=P(\tilde{\omega_{0}})$ for some Clifford subalgebra  $A\subset \R_{p,q}$, where $\omega_{0}$ is the Kahler form on $\R^{2m}$ defined by  by $J(e_{k})=e_{k+p}$, and $J(e_{k+p})=-e_{p}$ form $k=1,2,\ldots,m=\textrm{min}\{p,q\}$.  Moreover the image of the primitive idempotent $h=\tilde{f}\cdot e$ in  $\bigwedge^{\bullet}(\R^{p+q})^{*}$ via the extended symbol map splits as the product  $\sigma^{*}(h)=P^{\mathbb{Q}}(\tilde{\omega_{0}})\wedge\sigma^{*}(e)$. 

\end{proof}

%\begin{remark}
%On the Clifford algebra $\R_{q,q+l}$ for the quadratic space $\R^{2q+l}$ we may utilize  the projection $$Proj:\R^{2q+l}\rightarrow \R^{2q}$$ given by $(x^1,...x^{2q+l})\mapsto (x^{1},...,x^{q},x^{q+l+1},...,x^{q+l+j},..,x^{2q})$ to obtain $Proj^{*}(\sigma^{*}(f))=P(\omega_{0})$ in $\R^{2q}$. The projected Kahler form  $\omega_{0}$ is the Kahler form induced by the standard Euclidean inner product and the complex structure given by $J(e_{k})=e_{k+q}$, and $J(e_{k+q})=-e_{q}$ form $k=1,2,...,q$.  On the other hand for the Clifford algebra $\R_{p+l,p}$ for the quadratic space $\R^{2p+l}$ we utilize the projection $$Proj:\R^{2p+l}\rightarrow \R^{2p}$$ given by $(x^1,...x^{2p+l})\mapsto (x^{1},...,x^{p},x^{p+1},...,x^{p+j},..,x^{2p})$ to obtain $Proj^{*}(\sigma^{*}(f))=P(\omega_{0})$ in $\R^{2p}$, where $\omega_{0}$ is the Kahler form in $\R^{2p}$ induced by the standard Euclidean inner product and the complex structure given by $J(e_{k})=e_{k+p}$, and $J(e_{k+p})=-e_{p}$ form $k=1,2,...,p$. Hence by projecting onto smaller Euclidean spaces we can generate the $\textrm{U}(m)$ structures on these spaces from the Clifford algebras associated to larger quadratic spaces.

%\end{remark}

\section{Future research}

Having established a relationship between $\textrm{U}(n)$-structures and their associated minimal left ideals, we next plan to focus on the following questions: 

\begin{enumerate}
    \item How can we extend this identification between $\U(n)$-structures and minimal left ideals to spinor bundles on Kahler manifolds?
    \item Can we establish an identification between the quaternionic holonomy groups and minimal left ideals on Clifford algebras?
    \item Is there a way we can associate minimal left ideals with the holonomy groups of pseudo-Riemannian manifolds?
\end{enumerate}

\section{Acknowledgments }
We would like to thank Anna Fino for our discussion on future research directions for this work, and Jennifer Brown for discussions about the presentation of these ideas.

\end{document}